\documentclass[reqno,10pt,a4paper,oneside,draft]{amsart}
\usepackage{amssymb,amsmath,amsthm}
\usepackage{mathtools}
\usepackage{stmaryrd}
\usepackage{enumerate}
\usepackage[all,cmtip,2cell]{xy} 
\usepackage{fouridx}
\usepackage{bbm}
\usepackage{epigraph}

\usepackage[utf8]{inputenc}
\usepackage[usenames,dvipsnames]{xcolor}
\usepackage{geometry}
\usepackage{verbatim}
\usepackage{url}
\usepackage{cleveref}

\makeatletter
\renewenvironment{proof}[1][\proofname] {\par\pushQED{\qed}\normalfont\topsep6\p@\@plus6\p@\relax\trivlist\item[\hskip\labelsep\bf#1\@addpunct{.}]\ignorespaces}{\popQED\endtrivlist\@endpefalse}
\makeatother

\makeatletter

\renewcommand{\section}{\@startsection
{section}
{1}
{0mm}
{-\baselineskip}
{1\baselineskip}
{\Large \bfseries}}

\renewcommand{\subsection}{\@startsection
{subsection}
{2}
{0mm}
{-\baselineskip}
{0.5\baselineskip}
{\normalfont\large\bfseries}}

\makeatother

\def\defthm#1#2#3#4{
  \newtheorem{#1}[theorem]{#3}
  \newtheorem*{#1*}{#3}
  \newtheorem{#2}[theorem]{#4}
  \newtheorem*{#2*}{#4}
  \crefname{#1}{#3}{#4}
  \crefname{#2}{#4}{#4}  
}

\numberwithin{equation}{section}

\newtheoremstyle{mythm}%
{10pt}
{}
{\itshape}
{}
{\bf}
{.}
{.5em}
{}%

\newtheoremstyle{mydef}%
{10pt}
{3pt}
{}
{}
{\bf}
{.}
{.5em}
{}%

\newtheoremstyle{myrmk}%
{10pt}
{3pt}
{}
{}
{\bf}
{.}
{.5em}
{}%

\theoremstyle{mythm}

\newtheorem{theorem}{Theorem}[section]
\newtheorem*{theorem*}{Theorem}
\crefname{theorem}{Theorem}{Theorems}

\crefname{table}{Table}{Table}

\defthm{corollary}{corollaries}{Corollary}{Corollaries}
\defthm{lemma}{lemmata}{Lemma}{Lemmata}

\newtheorem*{replemmax}{\reptitle}
 {\end{replemmax}}

\defthm{proposition}{propositions}{Proposition}{Propositions}

\theoremstyle{mydef}
\defthm{definition}{definitions}{Definition}{Definitions}

\theoremstyle{myrmk}
\defthm{remark}{remarks}{Remark}{Remarks}
\defthm{example}{examples}{Example}{Examples}
\defthm{question}{questions}{Question}{Questions}

\crefname{section}{Section}{Sections}

\swapnumbers

\def\lam#1{{\lambda}\@lamarg#1:\@endlamarg\@ifnextchar\bgroup{.\,\lam}{.\,}}
\def\@lamarg#1:#2\@endlamarg{\if\relax\detokenize{#2}\relax #1\else\@lamvar{\@lameatcolon#2},#1\@endlamvar\fi}
\def\@lamvar#1,#2\@endlamvar{(#2\,{:}\,#1)}
\def\@lameatcolon#1:{#1}

\def\lamu#1{{\lambda}\@lamuarg#1:\@endlamuarg\@ifnextchar\bgroup{.\,\lamu}{.\,}}
\def\@lamuarg#1:#2\@endlamuarg{#1}

\newdir{ >}{{}*!/-7pt/@{>}}
\newdir{m}{->}

\newcommand{\pullback}[1]{\save*!/#1-1.2pc/#1:(-1,1)@^{|-}\restore}

\newcommand{\drpullback}{\pullback{dr}}

\newcommand{\ie}{\text{i.e.\ }}

\newcommand{\myemph}{\textit}

\newcommand{\defeq}{=_{\operatorname{def}}}
\newcommand{\co}{\colon}
\newcommand{\iso}{\cong} 

\newcommand{\op}{\operatorname{op}}

\newcommand{\SSet}{\mathbf{SSet}}

\newcommand{\UU}{\overline{\mathsf{U}}}
\newcommand{\U}{\mathsf{U}}
\newcommand{\VV}{\overline{\mathsf{V}}}
\newcommand{\V}{\mathsf{V}}

\newcommand{\Weq}{\mathsf{Weq}}
\newcommand{\Iseq}{\mathsf{Inv}}

\newcommand{\Linv}{\mathsf{LInv}}
\newcommand{\Rinv}{\mathsf{RInv}}
\newcommand{\Set}{\mathbf{Set}}

 \UseAllTwocells

\DeclarePairedDelimiter\braces\lbrace\rbrace

\newcommand{\Id}{\mathsf{Id}}

\newcommand{\hattimes}{\mathbin{\hat{\times}}}

\newcommand{\cal}[1]{\mathcal{#1}}

\newcommand{\Cof}{\mathsf{Cof}}
\newcommand{\TrivFib}{\mathsf{TrivFib}}
\newcommand{\Fib}{\mathsf{Fib}}
\newcommand{\TrivCof}{\mathsf{TrivCof}}

\newcommand{\kcyl}{\delta^k}

\makeatletter
\def\ignorespacesandallpars{%
  \@ifnextchar\par
    {\expandafter\ignorespacesandallpars\@gobble}%
    {}%
}
\makeatother

\newcommand{\note}[3][]{\def\auth{#1}\textcolor{#2}{{[\ifx\auth\empty\else\auth: \fi{#3}]}}}

\newcommand{\alert}[3][]{\def\auth{#1}\textcolor{#2}{{\ifx\auth\empty\else\auth: \fi{#3}}}}

\newcommand{\BFFib}{\mathsf{BFib}}

\newcommand{\Path}{\mathsf{Path}}

\title[]{Towards a constructive  simplicial model \\ of Univalent Foundations}

\begin{document}

\begin{abstract}
We provide a partial solution to the problem of defining a constructive version of 
Voevodsky's simplicial model of Univalent Foundations. For this, we prove constructive counterparts of the 
necessary results of simplicial homotopy theory, building on the constructive
version of the Kan-Quillen model structure established by the second-named author. In particular, 
 we show  that dependent products along fibrations with
 cofibrant domains preserve fibrations, establish the weak equivalence extension property
for weak equivalences between fibrations with cofibrant domain and define a univalent 
fibration that classifies small fibrations between bifibrant objects. 
These results allow us to define a comprehension category supporting identity types,
$\Sigma$-types, $\Pi$-types and a univalent universe, leaving only
a coherence question to be addressed.
\end{abstract}

\author{Nicola Gambino}
\address{School of Mathematics, University of Leeds, United Kingdom}
\email{n.gambino@leeds.ac.uk}

\author{Simon Henry}
\address{Department of Mathematics and Statistics, University of Ottawa, Canads}
\email{shenry2@uottawa.ca}

 \date{\today}

\maketitle

\tableofcontents

\section*{Introduction} 

\textbf{Context and motivation.} This paper investigates Voevodsky's model of Martin-L\"of type theory (ML)
extended with the Univalence Axiom (UA)  in the category of simplicial sets~\cite{voevodsky-simplicial-model}.  This model is of fundamental importance since it  informs a new approach to the development of mathematics, known as Univalent Foundations, which takes the notion of a space (rather than that of a set) as the most primitive and is supported by a new approach to computer-assisted formalisation of mathematics~\cite{voevodsky:library}. In combination with the discovery of a close connection between identity types and Quillen's homotopical algebra~\cite{awodey-warren:homotopy-idtype,gambino-garner:idtypewfs}, the simplicial model also provided inspiration for the development of  Homotopy Type Theory~\cite{hottbook}.

The original definition of the model  was carried out in  Zermelo-Fraenkel set theory extended with the axiom of choice (ZFC) and two inaccessible cardinals~\cite[Theorem~3.4.3]{voevodsky-simplicial-model}. Given the wide gap in proof-theoretic strength between that theory and ML~\cite{GrifforE:strsml} and the fact that the former is a classical theory  while the latter is a constructive one,  the question of whether the simplicial model could be defined working constructively (\ie without the use of the law of excluded middle or the axiom of choice) arose immediately after its discovery around~2006~\cite{KapulkinC:uniss,StreicherT:modtts,VoevodskyV:notts} and was one of the central issues investigated during the thematic programme on Univalent Foundations at the Institute for Advanced Study in 2012/13. In spite of  significant efforts since then, the question is still open to date.

The aim of this paper is to provide a partial solution to this problem.  In order to
explain our results and the novel aspects of our work, let us recall from~\cite{voevodsky-simplicial-model} that the main results of  homotopy theory necessary 
to define the simplicial model of ML(UA) = ML + UA are 
the existence of a Quillen
model structure on the category of simplicial sets whose fibrations are the Kan fibrations, 
 the fact that dependent product along a fibration preserves fibrations,
 the existence of a fibration $\pi \co \UU \to \U$ that classifies small fibrations, 
the fibrancy of $\U$ and, finally, the univalence of the fibration $\pi$. In combination, these results
allow us to define a comprehension category 
\[
\begin{gathered}
\xymatrix{
\mathbf{Fib}\,  \ar[dr] \ar@{>->}[rr] & & \SSet^{\to} \ar[dl]^{\mathrm{cod}} \\ 
 & \SSet &  }
 \end{gathered}
\]
where $\SSet$ is the category of simplicial sets 
and $\mathbf{Fib}$ is the category of Kan fibrations, 
and to show that this comprehension category supports the type
constructors of ML (which we take here to be $0$, $1$, $+$, $\mathbb{N}$, 
$\mathsf{Id}$, $\Sigma$, $\Pi$ and $\mathsf{U}$),
in the sense of~\cite{LumsdaineP:locuoc},  and a univalent type universe closed under the above type
constructors,  in the sense of~\cite{ShulmanM:allths}. Such a comprehension category is not quite a 
model of ML(UA) because of well-known strictness issues~\cite{HofmannM:intttl}, but it gives rise to one
by an appropriate coherence result~\cite{voevodsky-simplicial-model,LumsdaineP:locuoc}.

A fundamental obstruction to a constructive development of  the simplicial model was discovered in~\cite{coquand-non-constructivity-kan}, where it was shown that
it is not possible to prove constructively that for a simplicial set $A$ and a
Kan complex $B$, the exponential~$B^A$ is again a Kan complex.
Because of this, Coquand and his collaborators  defined homotopy-theoretic models of type
theories with the univalence axiom in categories of cubical sets~\cite{coquand-cubical-sets}, opening a profitable new research direction, cf.~\cite{awodey-cubical,cohen-et-al:cubicaltt,PittsAM:aximct} for example. Apart from switching 
from simplicial sets to cubical sets, they also switched from ordinary fibrations, defined by a right lifting property,
to uniform fibrations, defined as maps equipped with additional structure which 
 provides a choice of fillers for lifting problems satisfying uniformity conditions, as in algebraic weak factorisation systems~\cite{garner:small-object-argument,grandis-tholen-nwfs}. While 
categories of cubical sets considered in this line of work allow us to define constructively models of 
ML(UA) and admit a Quillen
model structure~\cite{SattlerC:equepu}, none of them
 is known to be Quillen equivalent to simplicial sets or topological spaces. Ongoing work of Awodey, Coquand, Riehl and Sattler  aims at addressing this issue 
 using equivariant presheaves in the cubical setting.

In simplicial sets, although it is possible to develop  a constructive theory of uniform fibrations and prove that dependent product along a uniform fibration preserves
uniform fibrations (and in particular that for a simplicial set $A$ and an 
uniform Kan complex $B$, the simplicial set $B^A$ is an algebraic Kan complex),
  as shown in~\cite{gambino2017frobenius}, the uniform fibrations
are not as well-behaved as in cubical sets,
since they do not admit a classifier, essentially because 
representables are not closed under under products in simplicial sets, unlike in cubical sets~\cite{SattlerC:faiaut}. 
In summary, to date there is no known model of ML(UA) that is definable in a constructive setting and
based on a category homotopically equivalent to that of simplicial sets. 

\smallskip

\noindent
\textbf{Main results.}
A breakthrough has been obtained  by the second-named author in~\cite{henry2019qms}, where, building on his previous work on weak model structures~\cite{henry2018wms},
it is shown constructively that the category of simplicial sets admits a Quillen model structure in which the
fibrations are the Kan fibrations. Moreover, this model structure is shown to be cartesian and proper.  Crucially for our goals here,  in this model structure not all monomorphisms are 
cofibrations, but only those monomorphisms $i \co A \to B$ that are levelwise complemented 
and for which the degeneracy of $n$-simplices is decidable in $B_n \setminus A_n$, for every $[n] \in \Delta$. In particular, not all simplicial sets are cofibrant, but only those with decidable degeneracies. However, this model structure coincides with
the standard one as soon as the law of excluded middle is assumed. Two other  
proofs of the  existence of this constructive
model structure are obtained in~\cite{GambinoN:anocp}.

 The present paper extends this work to obtain constructive counterparts of all the other main results of 
homotopy theory necessary to define the simplicial model of ML(UA). In particular, our main results are the following:
\begin{itemize}
\item \cref{thm:MainPathObject}, asserting that, for a fibration with cofibrant domain $p \co A \to X$, the mapping path space 
gives a factorisation of the diagonal $\delta_p \co A \to A \times_X A$ as a trivial cofibration followed by a fibration;
\item \cref{cor:Pi_types_are_fibrant}, asserting that dependent product along fibrations with cofibrant
domain preserves fibrations,  so that, for a cofibrant simplicial set $A$ and a Kan complex $B$, the exponential $B^A$ is a Kan complex; 
\item \cref{thm:universe-uc}, asserting the existence of a small fibration $\pi_c \co
\UU_c \to \U_c$ that classifies small fibrations between cofibrant objects;
\item \cref{thm:fibrancy-of-u-and-uc}, asserting that the simplicial set $\U_c$ is a cofibrant Kan complex;
\item \cref{thm:univalence-of-u-and-uc}, asserting that the fibration $\pi_c  \co
\UU_c \to \U_c$ is univalent.
\end{itemize}

These results allow us to define a comprehension category 
\begin{equation*}
\begin{gathered}
\xymatrix{
\mathbf{Fib}_{ \mathsf{cof}} \, \ar[dr] \ar[rr] & & \SSet^{\to}_{\mathsf{cof}} \ar[dl]^{\mathrm{cod}} \\ 
 & \SSet_{\mathsf{cof}} \, , &  }
 \end{gathered}
 \end{equation*}
 where $\SSet_{\mathsf{cof}}$ is the category of cofibrant simplicial sets and  $\mathbf{Fib}_{ \mathsf{cof}}$ is the category of fibrations between cofibrant simplicial
 sets, and to show that this comprehension category supports the type constructors of ML and a univalent type universe~(\cref{th:main_ContextualCat}). In this 
comprehension category, type-theoretic
contexts correspond to cofibrant objects, while dependent types correspond fibrations $p \co A \to X$ where 
 $A$ and $X$ are cofibrant. This choice is informed by the fact that, for a simplicial set~$X$, the slice category~$\mathbf{SSet}_{/X}$ admits a model structure in which the bifibrant objects are
 the fibrations $p \co A \to X$ with $A$ cofibrant. Our main results above show how this comprehension category  supports the type constructors of ML and a univalent type universe.
For example, \cref{thm:MainPathObject} shows that identity types can be interpreted as 
mapping path spaces. 

\smallskip

\noindent
\textbf{Novel aspects.} The key novelty of our approach is the use of the homotopy-theoretic notion of cofibrancy to encapsulate the logical notion of decidability of degeneracies and to work with it in a 
mathematically efficient way. Also, to the best of our knowledge, this paper is the first to make use of the cofibrant replacement functor in the study of models of dependent type theory.

We will use the cofibrant  replacement functor
to obtain $\Pi$-types and the type-theoretic universe in our comprehension category.  For $\Pi$-types, given a fibration $p \co A \to X$ with cofibrant domain, the result of applying the dependent product along $p$ to a fibration $q \co B \to A$ with cofibrant domain produces a map  $\Pi_p(q) \co \Pi_A(B) \to X$, which is a fibration by~\cref{cor:Pi_types_are_fibrant}, but whose domain
is not necessarily cofibrant. In order to remedy this, we define  $\Pi$-types to be given by a cofibrant replacement of
$\Pi_p(q) \co \Pi_A(B) \to X$, which is now a fibration with cofibrant domain. 
Interestingly, this definition validates
a judgemental $\beta$-rule and a propositional $\eta$-rule (see~\cref{rem:pi-types} for details), a combination that arises naturally when  Martin-L\"of type theory is presented
within the Logical Framework~\cite{nordstrom-petersson-smith:ml,GarnerR:strdp}, as well as in
the versions of the Coq assistant used for the original development of  univalent foundations
library~\cite{voevodsky:library}. Interestingly, Voevodsky's proof that the univalence axiom implies the principle of function extensionality exploits crucially the propositional $\eta$-rule for $\Pi$-types. 

For the definition of a  type-theoretic universe satisfying the univalence axiom, we proceed
in two steps. First, we construct a small fibration $\pi \co \UU \to \U$ that classifies small fibrations with cofibrant fibers and prove that $\U$ is fibrant. 
Then, we consider a cofibrant replacement~$\U_c$ of~$\U$, which comes equipped with a trivial fibration $\varepsilon \co \U_c \to \U$, so that 
we obtain a small fibration $\pi_c \co \UU_c \to \U_c$ with bifibrant codomain and cofibrant domain by pullback along $\varepsilon$. Our final result, \cref{thm:univalence-of-u-and-uc}, shows that
$\pi \co \UU \to \U$ and $\pi_c \co \UU_c \to \U_c$ are univalent fibrations.

We refrain from claiming that our work provides a complete solution to the problem of defining a constructive simplicial model of type theory since the comprehension category we construct is not split and therefore does not
satisfy strictly the stability conditions governing the interaction of substituion and type-formation rules, often
referred to as the Beck-Chevalley conditions. While there are several results to address similar issues~\cite{clairambaultdybjer2014,HofmannM:intttl,voevodsky-simplicial-model,LumsdaineP:locuoc,ShulmanM:allths}, they do not seem to apply in our  context. However, our results here reduce the problem of defining a constructive version of the simplicial model of Univalent Foundations to that of proving appropriate coherence result for a particular kind of 
comprehension category defined, which is independent of simplicial homotopy theory  and we leave to future work (see \cref{sec:conclusion} for
details).

Let us conclude these introductory remarks by mentioning that although our main results are inspired by the existing literature, especially those in~\cite{voevodsky-simplicial-model}, their proofs require systematic and careful cofibrancy considerations. In particular, in order to 
exhibit identity types as path objects (\cref{thm:MainPathObject}) and to prove the weak equivalence
extension property (\cref{Prop:Homotopy_ext_prop}) we will need to establish the non-trivial
and surprising fact that the appropriate dependent products preserve cofibrancy (see~\cref{prop:X^kCofibrant} and~\cref{Lemma:ForTheExtProperty} for details), a 
fact that does not hold in general. Furthermore, when discussing a universe, we are not interested in defining a small fibration that classifies small fibrations as in~\cite{voevodsky-simplicial-model}, but
rather a small fibration that classifies small fibrations between cofibrant objects.

\smallskip

\noindent 
\textbf{Outline of the paper.} \cref{sec:preliminaries} recalls the constructive version of the 
Kan-Quillen model structure and some auxiliary results from~\cite{henry2019qms}. \cref{sec:basrp} establishes some basic results on pullbacks. \cref{sec:pats} shows how path spaces provide factorisations
as trivial cofibration followed by fibration. \cref{sec:Pi-types} proves the restricted Frobenius property. \cref{sec:equep}
establishes the weak equivalence extension property.
\cref{sec:unifbb} introduces a fibration $\pi_c \co \UU_c \to \U_c$ that classifies small fibrations
between cofibrant Kan complexes. \cref{sec:fibrancy-and-univalence} proves that $\U_c$ is bifibrant and that $\pi_c \co \UU_c \to \U_c$ is univalent, using the weak equivalence extension property.
We conclude in~\cref{sec:conclusion} by summarising our main results in terms of
the comprehension category and stating some open problems for future research.

\smallskip

\noindent
\textbf{Remarks on constructivity.} We work in  Constructive Zermelo-Fraenkel set theory (CZF),
a set theory based on intuitionistic logic~\cite{AczelP:typtic-I}. When discussing smallness of sets,
simplicial sets and fibrations, we assume the existence of an inaccessible set $\mathsf{u}$, as defined 
in~\cite[Definition~18.1.2]{AczelP:notcst}. We say that a set is small if it is an element of $\mathsf{u}$.
For the closure of the type universe under $\Pi$-types, we will assume a form of `propositional resizing', asserting that every subset 
of a small set is again small, \ie  $\forall a \, , b \,  ( a \subseteq b \land b \in \mathsf{u} \rightarrow a \in \mathsf{u} )$.  A further discussion of what can be done without this assumption can be found in \cref{rem:strength}. 

Throughout the paper, we use a slight abuse of language and say that a map $f \co X \to Y$ has
the right lifting property with respect to a given set of maps to mean that we have a pair consisting
of $f \co X \to Y$ and a function assigning diagonal fillers to the corresponding set of lifting problems, with
no uniformity conditions. In particular, fibrations and trivial fibrations will always be treated as maps equipped with additional structure. Morphisms between such maps will always be taken to be mere commutative squares, with no compatibility conditions with respect to the chosen diagonal fillers. In contrast, cofibrations and
trivial cofibrations will be simply maps satisfying suitable properties.

\smallskip

\noindent 
\textbf{Acknowledgements.} Nicola Gambino gratefully acknowledges the support of
EPSRC under grant EP/M01729X/1 and the US Air Force Office for Scientific Research under 
agreement FA8655-13-1-3038 and the hospitality of the 
School of Mathematics of the University of Manchester and the Centre for Advanced Study in Oslo,
where the paper was written while on study leave from the University of Leeds. During the preparation of this work, Simon Henry was supported by the Operational Programme Research, Development and Education Project ``Postdoc@MUNI'' (No. CZ.02.2.69/0.0/0.0/16\_027/0008360). We are also grateful to Steve Awodey, Marc Bezem, Paolo Capriotti, Thierry Coquand, Andr\'e Joyal, Andy Pitts, Michael Rathjen, Christian Sattler, Thomas Streicher, Andrew Swan and Karol Szumi{\l}o  for useful conversations.

\section{Preliminaries} 
\label{sec:preliminaries}

We write $\Delta$ for the simplicial category. The objects of $\Delta$ are written as $[n]$, for $n \geq 0$.
We write $\SSet \defeq [\Delta^{\op}, \Set]$ for the category of simplicial sets. For $n \geq 0$, $\Delta[n] \in \SSet$ is the representable simplicial set associated to $[n] \in \Delta$. Given a map $f \co Y \to X$ in $\SSet$, we write $f^* \co \SSet_{/X} \to \SSet_{/Y}$
for the associated pullback functor. The functor $f^*$ has a left adjoint, defined
by composition, which we write $\Sigma_{f} \co \SSet_{/Y} \to \SSet_{/X}$ and refer
to as the \emph{dependent sum} along $f$. Since $\SSet$ is locally cartesian closed, pullback
along $f$ has also a right adjoint, which we write 
$\Pi_f \co \SSet_{/Y} \to \SSet_{/X}$ and refer to as the \emph{dependent
product}  along $f$. The action of these functors on a
map $g \co Z \to Y$ will be written $\Sigma_f(g) \co \Sigma_Y(Z) \to X$ and 
$\Pi_f(g) \co \Pi_Y(Z) \to X$, respectively. Since  $\Sigma_f$ is defined by composition,~$\Sigma_Y(Z) \defeq Z$ and~$\Sigma_f(g) = fg $. 

As a special case of a well-known result for presheaf 
categories,  for every~$[n] \in \Delta$ there is an equivalence of categories
\begin{equation}
\label{equ:psh-slice-sset}
\SSet_{/\Delta[n]} \simeq  [ {\Delta_{/[n]}}^{\op}, \Set]   \, .
\end{equation}
For $F \co {\Delta_{/[n]}}^{\op} \to \Set$, we write $\pi_F \co \int F \to \Delta[n]$
for the corresponding object of~$\SSet_{/ \Delta[n]}$. Here, $\int F$ is the
simplicial set whose $m$-simplices are pairs $(\theta, x)$ where $\theta \co [m] \to [n]$
is a map in~$\Delta$ and $x \in F(\theta)$. The components of the natural
transformation $\pi_F$ are the first projections.

For $[n] \in \Delta$, we write $i^n \co  \partial \Delta[n] \to \Delta[n]$ for the boundary inclusion into the $n$-simplex and, for $1 \leq i \leq n$,  
$h^i_n  \co \Lambda^i[n] \to \Delta[n]$  for the $i$-th horn inclusion into the $n$-simplex. The simplicial set $\Delta[1]$ is an interval object in $\SSet$, with endpoint inclusions~$\kcyl \co \braces{ k} \to \Delta[1]$ defined by~$\kcyl \defeq h^1_k$. Throughout this paper, we shall work  with the constructive 
Kan-Quillen model structure $(\Weq, \Cof, \Fib)$ on $\SSet$ defined in~\cite{henry2019qms}. 
For the convenience of the reader, we recall the main aspects of this model structure below.
For this, let $\TrivCof \defeq \Weq \cap \Cof$ and $\TrivCof \defeq \Weq \cap \Fib$ be the classes of trivial cofibrations and trivial fibrations,
respectively. 

The weak factorisation system $(\Cof, \TrivFib)$ of cofibrations and trivial fibrations  is cofibrantly generated by the set $\cal{I} \defeq \{ i^n \co  \partial \Delta[n] \to \Delta[n] \ | \ n \geq 0 \}$  of boundary 
inclusions, \ie 
\[
\big( \Cof, \TrivFib \big) = ( \mathsf{Sat}(\cal{J}) \, , \cal{J}^\pitchfork) \, .
\]
As shown in \cite[Proposition~5.1.4]{henry2018wms} a map $f \co Y \to X$ is a cofibration if  and only if 
it is a levelwise complemented monomorphism and the degeneracy of the simplices of $X_n ~\setminus~Y_n$ is decidable for every $[n] \in \Delta$. Thus, a simplicial set $X$ is
cofibrant if  degeneracy of the simplices of $X$ is decidable.
Note that a map between cofibrant objects is a cofibration
if and only if it is a levelwise complemented monomorphism. 
Cofibrant simplicial sets are of particular importance for our development because of their decidability property, which can be used to establish counterparts of classical results valid for all simplicial sets. An example is the Eilenberg-Zilber lemma~\cite[Lemma~5.1.2]{henry2018wms}, asserting that in a cofibrant simplicial set~$X$, any cell $x \in X$ can be written uniquely as $p^*(y)$, where $y$ is a non-degenerate cell of $X$ and $p$ is a degeneracy. 
\medskip

The weak factorisation system $(\TrivCof, \Fib)$ of  trivial cofibrations and fibrations is cofibrantly 
generated by the set $\cal{J} \defeq \{ h^k_n  \co \Lambda^k[n] \to \Delta[n]  \ | \ 0 \leq k \leq n \}$ of horn 
inclusions, \ie 
\[
(\TrivCof, \Fib) = ( \mathsf{Sat}(\cal{J}) \, , \cal{J}^\pitchfork) \, .
\] 
For a map $f \co Y \to X$, we denote the action of the pullback $f^*$ on a fibration~$p \co A \to X$ as
\[
\xymatrix{
A[f] \ar[r] \drpullback \ar@{->>}[d]_{p[f]} & A \ar@{->>}[d]^{p} \\
Y \ar[r]_f & X \, .}
\]

For $X \in \SSet$, we write $\Fib_{/X}$ for the category whose objects are fibrations with codomain~$X$ 
and whose maps are determined by saying that the functor $U \co \Fib_{/X} \to \SSet_{/X}$ forgetting the
fibration structure is  full and faithful. We then write $\BFFib_{/X}$ for the full subcategory of $\Fib_{/X}$  spanned by fibrations with cofibrant domain. For a simplicial set $X$, we write $\mathbb{L}(X)$ for its cofibrant replacement and $\mathbb{R}(X)$ for its
fibrant replacement. These objects come equipped with a trivial fibration $\varepsilon_X \co \mathbb{L}(X) \to X$ and a trivial cofibration $\eta_X \co 
X \to \mathbb{R}(X)$, respectively. An explicit definition of the cofibrant replacement is given 
in Appendix~\ref{sec:appendix}.

\medskip

The model structure $(\Weq, \Cof, \Fib)$ is proper, \ie both left and right proper~\cite[Propositions 2.2.9 and 3.5.2]{henry2019qms}
and the weak factorization systems $(\Cof, \TrivFib)$ and $(\TrivCof, \Fib)$
satisfy the so-called pushout product property~\cite[Proposition 5.1.5 and Corollary 5.2.3]{henry2018wms}. Recall that, given two maps $f \co Y \rightarrow X$ and $g \co B \rightarrow A$, their \emph{pushout product} $f \hattimes g$ is defined as the unique dotted map in the diagram
\[
\xymatrix{
Y \times B \ar[r] \ar[d] &  X \times B \ar[d] \ar@/^1.5pc/[ddr] \\
Y \times A \ar[r]  \ar@/_1.5pc/[drr] & \displaystyle \big( Y \times A ) +_{Y \times B} \big( X \times B ) \ar@{.>}[dr]  \\
 & & X \times A \, . }
 \]
The pushout-product property is  the statement that  if $f$ and $g$ are cofibrations then so is $f \hattimes g$
and, if additionally either $f$ and $g$ is a weak equivalence, then so is $f \hattimes g$.
Note that when~$f$ and~$g$ are monomorphisms, so in particular when they are cofibrations, the pushout in the diagram above is just an union of subobjects and the pushout product of~$f$ and~$g$ is just the inclusion
\[  
f \hattimes g \co (Y \times A) \cup (X \times B) \rightarrow X \times A \, .
\]
Dually, using  exponentials instead of products and pullbacks instead of pushouts, for maps $f \co Y \rightarrow X$ and~$p \co B \rightarrow A$, the \myemph{pullback exponential}  $\langle f \, , p \rangle$ is defined as the unique dotted arrow in the diagram:
\[
\xymatrix{
 B^X \ar@{.>}[dr] \ar@/^1.5pc/[drr] \ar@/_1.5pc/[ddr] \\
& B^Y \times_{A^Y} A^X \ar[r] \ar[d] &  A^X \ar[d]  \\
& B^Y \ar[r] & A^Y  \, . \\
 }
 \]
By adjointness (see~\cite{joyal-tierney-segal} for details), $\langle p \, , q \rangle $ has the right lifting property against a map $f$  if and only if $q$ has the right lifting property agains $f \hattimes p$.
Therefore, the pushout-property  implies its dual version, asserting that  if $f \co Y \to X$ is a cofibration and $p \co B \to A$ a fibration then $\langle f \, , p \rangle$ is a fibration and if, additionally,  either $p$ or $f$ is a weak equivalence, then so so is $\langle f \, , p \rangle$.

\medskip

We also make use of another weak factorisation system $(\mathsf{L}, \mathsf{R})$ on $\SSet$ introduced in~\cite[Section~3.1]{henry2019qms}. This is useful to establish decidability conditions; in particular, we will use it to prove~\cref{proposition:PathObjectCofibrant} and~\cref{Lemma:ForTheExtProperty} below. By definition, $(\mathsf{L}, \mathsf{R})$ is the
weak factorisation system cofibrantly generated by the set of degeneracy maps $\sigma \co \Delta[m] \to \Delta[n]$, \ie the maps induced by surjections $[m] \rightarrow [n]$ in $\Delta$. We refer to maps in $\mathsf{L}$ as the \emph{degeneracy quotients} and the
maps in $\mathsf{R}$ as the \emph{degeneracy-detecting} maps. Since degeneracy maps are (split) epimorphisms in $\SSet$, the weak factorization system $(\mathsf{L}, \mathsf{R})$ is actually a unique factorization system. Since the degeneracy-detecting maps are the maps with the (unique) right lifting property against degeneracy map,  they are exactly the simplicial morphisms $f \co X \rightarrow Y$ such that for $x \in X$, $f(x)$ is degenerated if and only if $x$ is degenerated. The degeneracy quotient maps are
instead the pushouts of coproducts of degeneracy maps, \ie the maps of the form $X \rightarrow X[(x_i,\sigma_i])$ where $x_i \in X_{n-1}$ are a family of cells,  $\sigma \co [n_i] \twoheadrightarrow [m_i]$ a family of degeneracies and $X[(x_i,\sigma_i)]$ is obtained from~$X$ by freely making $x_i$ in the image of $\sigma_i $ for each $i$. For the convenience of the reader, we recall from~\cite{henry2019qms} 
 the results  about this factorization system that will be used here.

\begin{lemma}[{\cite[Lemma~3.1.8]{henry2019qms}}]
\label{lem:decidability_lift_degen_quo}
Let $p \co A \rightarrow B$ be a degeneracy quotient between finite decidable simplicial set, $X$  a cofibrant simplicial set and $f \co
A \rightarrow X$ be any morphism. Then it is decidable whether $f$ can be factored through $p$ or not.  \qed
\end{lemma}

The proof of this result above can be outlined in a few words. Since $B = A[(a_i,\sigma_i)]$ for a finite collection of cells $a_i$,  $f$ factors through $p$ if and only if $f(a_i) =\sigma_i^*  x_i$ for each $i$, but when $X$ is cofibrant one can decide for each $i$ if this is the case or not. As there is only a finite number of such $i$, one can decide if it is the case for all $i$ or not. See

\begin{proposition}[{\cite[Proposition 3.1.11]{henry2019qms}}]
\label{prop:Degen_quotient_are_pullback_stable}
The class of degeneracy quotients is stable under pullback.  \qed
\end{proposition}

By~\cref{prop:Degen_quotient_are_pullback_stable}, if $f \co A \rightarrow B$ is a degeneracy quotient  then $f \times X \co A \times X \rightarrow B \times X$ is again a degeneracy quotient for every $X$.

\section{The comprehension category of cofibrant simplicial sets}
\label{sec:basrp}

The main goal of this section is to introduce the comprehension category of cofibrant simplicial sets. The
results of simplicial homotopy theory that we obtain in the following sections will show how this 
comprehension category supports various type-theoretic constructions. We refer to~\cite{JacobsB:catltt} for the definition of a comprehension category and basic results and to~\cite{LumsdaineP:locuoc} for the
definitions of the categorical counterparts of the type-theoretic constructs.

\begin{lemma}\label{lem:cofibrant_fiber_product} \hfill 
\begin{enumerate}[$(i)$] 
\item Let $A \, , B$ be cofibrant simplicial sets. Then their product $A \times B$ is cofibrant.
\item Let $f \co A \to X$ and $g \co B \to X$ be maps with  cofibrant domain. Then their
fiber product $A \times_X B$, fitting in the pullback diagram
\[
\xymatrix{
A \times_X B \drpullback \ar[r]^-{q} \ar[d]_-{p} & B \ar[d]^g \\
A \ar[r]_f & X \, ,}
\]
is cofibrant.
\end{enumerate}
\end{lemma}

\begin{proof} Part~(i) follows immediately from the fact that the model structure on simplicial sets is cartesian, i.e. it satisfies the pushout-product conditions.


For part~(ii), $A \times_X B$ is a sub-simplicial set of $A \times B$ hence a cell of~$A \times_X B$ is degenerate if and only if it is degenerate as a cell of $A \times B$, hence degeneracy
in~$A \times_X B$ is indeed decidable. 
\end{proof}

The next proposition introduces the  comprehension category  cofibrant simplicial sets.
Recall that we write~$\SSet_{\mathsf{cof}}$ for the full subcategory of  $\SSet$ spanned by
cofibrant simplicial sets. We then define $\mathbf{Fib}_{ \mathsf{cof}}$ to be the category that 
has fibrations between cofibrant objects as objects and commutative squares as maps. Here,
recall from the Introduction that we take fibrations to be pairs consisting of a map
and a function providing fillers for lifting problems against horn inclusions.

\begin{proposition} \label{thm:compcat}
The category $\SSet_{\mathsf{cof}}$ is the base of a comprehension category
of the form
\begin{equation*}
\begin{gathered}
\xymatrix{
\mathbf{Fib}_{ \mathsf{cof}} \ar[dr]_{p} \ar[rr]^{\chi} & & \SSet^{\to}_{\mathsf{cof}} \ar[dl]^{\mathrm{cod}} \\ 
 & \SSet_{\mathsf{cof}} \, , &  }
 \end{gathered}
 \end{equation*}
 where $\chi$ is the evident forgetful functor.
\end{proposition} 

\begin{proof} The category $\SSet^\to_{\mathsf{cof}}$ is the arrow category of $\SSet_{\mathsf{cof}}$, whose objects are
maps  between cofibrant objects. Then, the codomain
functor 
\[
\mathrm{cod} \co \SSet^{\to}_{\mathsf{cof}} \to \SSet_{\mathsf{cof}}
\] 
is a Grothendieck fibration by part~(ii) of~\cref{lem:cofibrant_fiber_product}. We then define $p = \mathrm{cod} \circ \chi$, which
 is also a Grothendieck fibration, since the pullback of a fibration (with chosen lifts) is a fibration (with chosen lifts).
  Since $\chi$ is the forgetful functor, it clearly preserves pullbacks.
\end{proof}

\begin{remark} \label{thm:sigma-types} It is immediate to see that the comprehension category of cofibrant simplicial sets supports $\Sigma$-types. since for a fibration $p \co A \to X$ between cofibrant objects, dependent sum along $p$  functor~$\Sigma_p \co \SSet_{/A} \to \SSet_{/X}$, which is defined by
composition with $p$, maps fibrations to fibrations.
\end{remark}

\medskip

We conclude this section with some  results on cofibrant objects and cofibrations
that will be useful in the following.

\begin{proposition} \label{thm:cof-pbk}  Let $p \co A \to X$  be a map with cofibrant domain.
Then pullback along $p$, \[
p^* \co \SSet_{/X} \to \SSet_{/A} \, , 
\]
preserves cofibrations. 
\end{proposition}

\begin{proof} Let $f \co Y \to X$ be a cofibration and consider the pullback
\[
\xymatrix{
A[f]  \ar@{>->}[r] \drpullback \ar[d]_-{p[f]} &  A \ar[d]^{p} \\
Y \ar@{>->}[r]_{f} \ & X \, .}
\]
We need to show that the monomorphism $A[f] \rightarrowtail A$ is a cofibration. 
Since $A$ is cofibrant, it suffices to show that it is a levelwise complemented. For $a \in A_n$, we have that $a \in A[f]_n$ if and only if $p(a) \in Y_n$. Since $f \co Y \rightarrowtail X$ is a levelwise complemented monomorphism, this is decidable.
\end{proof} 

Here and thoughout the paper, a {\em finite set} means a set equipped with a bijection with $\{1 \, , \ldots,  n \}$, for some $n \in \mathbb{N}$. Equality in such a set
is always decidable, which we sometimes stress by saying `finite decidable set', rather than just `finite set' to distinguishes it from other notion of finite set, like Kuratowski finiteness.

We say that a simplicial set $X$ is finite if $X([n])$ is a finite decidable set for each $n$ and if $X$ has a finite number of non-degenerate cell in total.

\begin{lemma} \label{prop:X^kCofibrant}  \hfill 
\begin{enumerate}[$(i)$]
\item Let $X$ be cofibrant and $K$ finite decidable simplicial set. Then $X^K$ is cofibrant.
\item Let $f \co A \rightarrow B$ be a degeneracy quotient of finite decidable simplicial sets and $X$  cofibrant. Then $X^f \co X^B \rightarrow X^A$ is a cofibration between cofibrant objects.
\end{enumerate}
\end{lemma}

\begin{proof}  For part~(i), recall that an $n$-cell $\Delta[n] \rightarrow X^K$ is a morphism $\Delta[n] \times K \rightarrow X$. Now let $\sigma \co \Delta[n] \rightarrow \Delta[m]$ be a degeneracy map. Then the map 
$\sigma \times K \co \Delta[n] \times K \rightarrow \Delta[m] \times K$ is again degeneracy quotient.
Since it is a map between finite decidable simplicial sets, the question of whether a map $\Delta[n] \times K \rightarrow X$ factors through $\Delta[n] \times K \rightarrow \Delta[m] \times K$ is decidable
by~\cref{lem:decidability_lift_degen_quo}. But this is exactly the question of degeneracy of cells in $X^K$.
Similarly, for part (ii), since an $n$-cell of $X^A$ is a morphism $\Delta[n] \times A \rightarrow X$, it belongs to $X^B$ if and only it can be factored as $\Delta[n] \times A \rightarrow \Delta[n] \times B \rightarrow X$. This is decidable by \cref{lem:decidability_lift_degen_quo} because $\Delta[n] \times A \rightarrow \Delta[n] \times B$ is a degeneracy quotient between finite decidable simplicial sets by part (i).
\end{proof}

\section{Identity types as path spaces}
\label{sec:pats}

In this section we begin to show how the the comprehension category of cofibrant simplicial
sets introduced in \cref{thm:compcat} supports 
various type-theoretic constructs by considering identity types. Following the fundamental insight in~\cite{awodey-warren:homotopy-idtype}, in order to equip the comprehension category  with identity types it suffices to show that
for every fibration $p \co A \to X$ with cofibrant domain, there are a  fibration with cofibrant domain~$\partial \co \Id_A \to A \times_X A$ and a trivial cofibration $\mathsf{refl}_A \co A \to \Id_A$ that provide a factorisation of the diagonal map $\delta_p \co A \to A \times_X A$~\cite{LumsdaineP:locuoc}. While such a factorisation is guaranteed to exist by the model structure on $\SSet$, here we show that it can be obtained by letting~$\Id_A$ be a mapping path space. 
This is useful in order to support the intuition of elements of identity types as paths, which is a central idea in Homotopy Type Theory~\cite{hottbook}, to ensure pullback stability
of the construction, as well as preservation of (fibrewise) smallness.

\begin{lemma} \hfill 
 \label{thm:exponentials}
\begin{enumerate}[(i)] 
\item Let $X$ be cofibrant and $A$ be fibrant.  Then $A^X$ is fibrant.
\item Let $f \co Y \rightarrow X$ be a cofibration and $A$ be fibrant. Then $A^f \co A^X \rightarrow A^Y$ is a fibration.
\item Let $f \co Y \rightarrow X$ be a trivial cofibration and $A$ be fibrant.  Then $A^f \co A^X \rightarrow A^Y$ is a trivial fibration.
\item Let $X$ be cofibrant and $p \co B \rightarrow A$ be a (trivial) fibration. Then $p^X \co B^X \rightarrow A^X$ is also a (trivial) fibration.
\end{enumerate}
\end{lemma}

\begin{proof} The claims follow easily from the pushout product property of the model structure.
\end{proof}
 
Interestingly, the cofibrancy assumptions of part~(i) of~\cref{thm:exponentials} allows
us to prove the claim constructively also by following the combinatorial proof in~\cite{MayJP:simoat}, exploiting the decidability of degeneracy in $X$ instead of appealing to the law of excluded middle.

For a simplicial set $A$, we define its path object by letting $\Path(A) \defeq A^{\Delta[1]}$. 
There are evident boundary map $\partial \co \Path(A) \to A \times A$, giving
the endpoints of a path. We write $\partial_0$, $\partial_1 \co \Path(A) \to A$ for the compositions of $\partial$
with the two projections and $r \co A \to \Path(A)$ for the `constant path' map.

\begin{proposition} \label{thm:id-types-for-types}
Assume that $A$ is fibrant. Then,
\begin{enumerate}[(i)] 
\item $\Path(A)$ is fibrant,
\item the boundary map $\partial  \co \Path(A) \rightarrow A \times A$ is a fibration,
\item the composite of $\partial \co \Path(A) \rightarrow A \times A$ with either projection is a trivial fibration,
\item the map $r \co A \rightarrow \Path(A)$ induced by the unique map $\Delta[1] \rightarrow \Delta[0]$ is a weak equivalence.
\end{enumerate}
\end{proposition} 

\begin{proof}
Part~(i) is just a special case of part~(i) of \cref{thm:exponentials}. For part~(ii), apply part~(ii) of \cref{thm:exponentials} to the cofibration $i^1 \co \partial \Delta[1]  \hookrightarrow \Delta[1]$. For part~(iii), apply part~(iii) of \cref{thm:exponentials} to the horn inclusions $h^k_n \co \Lambda^k[1]  \rightarrow \Delta[1]$. Part~(iv) follows from the 3-for-2 property for weak equivalences applied to $A \rightarrow A^{\Delta[1]} \rightarrow A$. Indeed, the
composite is the identity and the second factor has just been proved to be a trivial fibration.
\end{proof}

In order to interpret identity types, we need $\Path(A)$ to be cofibrant and the map $r \co A \rightarrow \Path(A)$ to be a trivial cofibration. This is achieved by the following proposition.

\begin{proposition}\label{proposition:PathObjectCofibrant}
Let $A$ be cofibrant. Then $\Path(A)$ is cofibrant and the map~$r \co A \rightarrow \Path(A)$ is a cofibration.
\end{proposition}

\begin{proof}
This follows from part~(iii) of \cref{prop:X^kCofibrant}, applied to the cofibrant simplicial set $X$, and the degeneracy map between finite simplicial sets $\Delta[1] \rightarrow \Delta[0]$. 
\end{proof}

We now define mapping path spaces and  extend~\cref{thm:id-types-for-types} and~\cref{proposition:PathObjectCofibrant}. Given a map $p \co A \to X$, we define $\Path(p)$ 
via the pullback diagram
\[
\xymatrix{
\Path(p) \drpullback \ar[r] \ar[d] & \Path(A) \ar[d] \\
X \ar[r]_-{r_X} & \Path(X) }
\]
The structural maps $r_p \co A \rightarrow \Path(p)$ and $\partial \co \Path(p) \rightarrow A \times_{X} A$ are produced by the diagram:
\[
\xymatrix{
& A \ar[rr] \ar[dd] & & \Path(A) \ar[rr] \ar[dd] & & A \times A \ar[dd] \\
A \ar[ur] \ar[rr]^(.6){r}  \ar[dd] & & \Path(p) \ar[ur] \ar[rr]^(.65){\partial} \ar[dd] & & A \times_{X} A \ar[ur] \ar[dd] \\
& X \ar[rr] & & \Path(X) \ar[rr] & & X \times X \, , \\
X \ar[ur] \ar[rr] & & X \ar[rr] \ar[ur] & & X \ar[ur] \\ 
}
\]
where the three square in the vertical/diagonal direction are pullbacks. As before, the maps
$\partial_0 \, , \partial_1 \co \Path(A) \to A$ are defined the composites of $\partial$
with the two projections.

\begin{theorem}
\label{thm:MainPathObject}
Assume that $p \co A \rightarrow X $ is a fibration between cofibrant objects Then,
\begin{enumerate}[(i)] 
\item \label{thm:MainPathObject:IdBifib} $\Path(p)$ is cofibrant, 
\item the map $\Path(p) \rightarrow X$ is a fibration,
\item the map $\partial \co \Path(p) \rightarrow A \times_{X} A$ is a fibration,
\item $\partial_k \co \Path(p) \rightarrow A \times_{X} A$  is a trivial fibration, for $k \in \braces{0 \, , 1 }$,
\item the map $r_p \co A \rightarrow \Path(p)$ is a trivial cofibration.
\end{enumerate}
\end{theorem}

\begin{proof} The map $\Path(p) \rightarrow X$ is a pullback of the maps $\Path(A) \rightarrow \Path(X)$ along $X \rightarrow \Path(X)$. Hence, since $\Path(A) \rightarrow \Path(X)$ is a fibration (by part~(iv) of~\cref{thm:exponentials}), the map $\Path(p) \rightarrow X$ is a fibration. Since $X$ is cofibrant by assumption and $\Path(A)$ is  cofibrant by \cref{proposition:PathObjectCofibrant}, we have that $\Path(p)$ is also cofibrant by~\cref{lem:cofibrant_fiber_product}. 

By the dual of the pushout-product property, the map $\langle \partial \Delta[n] \hookrightarrow \Delta[n] ,  A \rightarrow X \rangle$ is a fibration. This map is 
\[ 
\Path(A) \rightarrow (A \times A) \times_{X \times X} \Path(X) \, .
\] 
Moreover, in the diagram
\[
\xymatrix{
\Path(p) \drpullback \ar[r] \ar[d] & A \times_{X} A \ar[d] \ar[r] \drpullback & X \ar[d] \\
\Path(A) \ar[r] & \Path(X) \times_{X \times X}  (A \times A) \ar[r]  & \Path(X) \, ,
}
 \]
the right hand square is easily seen to be a pullback and the total rectangle is the pullback defining $\Path(p)$, hence the left hand square is also a pullback. Since the bottom left map is a fibration,  $\Path(p) \rightarrow A \times_{X} A$ is a fibration as well.

By a similar argument, for $k \in \braces{ 0 \, , 1}$, the map $\langle \Lambda^k[n] \hookrightarrow \Delta[n] \, , A \rightarrow X\rangle$ is a trivial fibration. Indeed, this is the map $q \co \Path(A) \rightarrow  A  \times_X \Path(X)$ which fits into the pullback diagrams
\[
\xymatrix{
\Path(p) \drpullback \ar[r] \ar[d] & A  \ar[d] \ar[r] \drpullback & X \ar[d] \\
\Path(A) \ar[r]_-{q} & A \times_X \Path(X) \ar[r]  & \Path(X) \, .
}
 \]
This shows that the canonical maps $\partial_k \co \Path(p) \rightarrow A$ are trivial fibrations.

We conclude by showing that the map $A \rightarrow \Path(p)$ is levelwise complemented. Indeed, it fits into a factorization 
\[
A \rightarrow \Path(p) \rightarrow \Path(A)
\] 
of a map which has been proved to be a levelwise complemented inclusion in~\cref{proposition:PathObjectCofibrant}. Therefore, for any cell of $\Path(p)$ one can decide if it is in $A$ or not by considering it as a cell of $\Path(A)$. Since $A$ and $\Path(p)$ are cofibrant, this shows that $A \rightarrow \Path(p)$ is a cofibration. The 3-for-2 property for weak equivalences applied to $A \rightarrow \Path(p) \rightarrow A$ shows that  $A \rightarrow \Path(p)$ is also a weak equivalence, and hence a trivial cofibration.
\end{proof}

\section{$\Pi$-types via cofibrant replacements}
\label{sec:Pi-types}

The aim of this section is to prove the results of simplicial homotopy theory necessary to show that
the comprehension category of cofibrant simplicial sets supports $\Pi$-types. 
In order to do this, we should consider a fibration with cofibrant domain $p \co A \to X$  and define an operation mapping a fibration with cofibrant domain $q \co B \to A$  to a new fibration with cofibrant domain~$\widetilde{\Pi}_p(q) \co \widetilde{\Pi}_A(B) \to X$ together with additional data~\cite{LumsdaineP:locuoc}. 
Given such a map $q \co B \to A$, we proceed
in two steps. First, we apply the dependent product along $p$, 
\[
\Pi_p \co \SSet_{/A} \to \SSet_{/X} \, , 
\]
to $q \co B \to A$ and obtain a map $\Pi_p(q) \co \Pi_A(B) \to X$, which we will show to be again a fibration. Since this fibration does not seem to have  cofibrant domain in general, we then apply a
cofibrant replacement in $\SSet_{/X}$ to $\Pi_p(q) \co \Pi_A(B) \to X$ so as to obtain a map with all the desired properties. Interestingly, the result will support a categorical counterpart of the propositional
$\eta$-rule for $\Pi$-types.

Below, we show that, for a fibration $p \co A \rightarrow X$ with cofibrant domain, the dependent product along $p$ preserves fibrations. By adjointness, this is equivalent to showing that its left adjoint, \ie 
 pullback along $f$, preserves trivial cofibrations. This amounts to proving a restricted version of the Frobenius property~\cite{BergB:topsmi}, obtained by considering pullbacks along fibrations with
cofibrant domain rather than general fibrations. By the obstruction results in~\cite{coquand-non-constructivity-kan}  these cofibrancy assumptions are essential in our constructive setting.

\begin{theorem}\label{cor:Pi_types_are_fibrant}
Let $p \co A \rightarrow X$ be a fibration with cofibrant domain,
\begin{enumerate}[$(i)$]
\item The pullback functor $p^* \co \SSet_{/X} \to \SSet_{/A}$ preserves trivial cofibrations.
\item The dependent product $\Pi_p \co \SSet_{/A} \to \SSet_{/X}$ preserves  fibrations.
\end{enumerate}
\end{theorem}

\begin{proof} Part~(i) follows from \cref{thm:cof-pbk}, which shows that $p^*$ preserves cofibrations, together with the right properness of the simplicial model structure which shows that it preserves weak equivalences.
 Part~(ii) follows from part~(i) as $\Pi_p$ is right adjoint to $p^*$.
\end{proof}

Another proof of~\cref{cor:Pi_types_are_fibrant} is given in~\cite[Section~4]{GambinoN:anocp} modifying appropriately the arguments in~\cite{gambino2017frobenius}.

\begin{remark}  \label{rem:pi-types}
We define explicitly how the $\Pi$-types of the comprehension category of cofibrant simplicial sets are defined. For this, let us recall that, for maps $p \co A \to X$ and $q \co B \to A$,  the dependent product $\Pi_p(q) \co \Pi_A(B)
\to X$ is equipped with a map
\[
\mathsf{app} \co \Pi_A(B) \times_A  A \to B
\] 
in $\SSet_{/A}$ which is universal in the sense that, for every  $Y \to X$, the function
\[
\begin{array}{rcl} 
 \SSet_{/X}[ Y , \Pi_A(B)] & \longrightarrow &  \SSet_{/A}[Y \times_A A, B]  \\
  f & \longmapsto & \mathsf{app}(f \times_A 1_A) 
  \end{array} 
 \]
 is a bijection. This means that we have a function $\lambda$ in the opposite direction such that  
 \begin{equation}
 \label{equ:betaeta}
 \mathsf{app}(\lambda(b) \times_A 1_A) = b   \, , \quad
 \lambda( \mathsf{app}(f \times_A 1_A)) = f \, ,
 \end{equation}
 for every $b \co Y \times_A A \to B$ and $f \co Y \to \Pi_A(B)$.  These equations correspond to the
 well-known judgemental $\beta$-rule and $\eta$-rule for $\Pi$-types, respectively.
 
 When $p$ and $q$ are fibrations and $A$ is cofibrant, the map 
 $\Pi_p(q) \co \Pi_A(B) \to X$ is a fibration by \cref{cor:Pi_types_are_fibrant} but $\Pi_A(B)$ is not cofibrant,
 in general. Thus, we interpret  $\Pi$-types as the 
 cofibrant replacement  of $\Pi_A(B)$, which is given by a cofibrant simplicial set
 $\mathbb{L}(\Pi_A(B)$  equipped with
 a trivial fibration $\varepsilon \co \mathbb{L}(\Pi_A(B)) \to \Pi_A(B)$. 
We then define $\widetilde{\mathsf{app}} \co   \mathbb{L}(\Pi_A(B)) \times_A A \to B$ as the composite
\[
\widetilde{\mathsf{app}}  \defeq \mathsf{app} \circ (\varepsilon \times_A 1_A) \, .
\]
For a bifibrant simplicial set $Y$ and maps $Y \to X$,  $b \co Y \times_A A \to B$, we define $\widetilde{\lambda}(b) \co Y \to \mathbb{L}(\Pi_A(B))$ to be the
diagonal filler
\begin{equation}
\label{equ:diag-for-beta}
\begin{gathered}
\xymatrix{
0 \ar[r] \ar[d] & \mathbb{L}(\Pi_A(B))  \ar[d]^\varepsilon \\
Y \ar[r]_-{\lambda(b)} \ar@{.>}[ur] & \Pi_A(B) \, ,}
\end{gathered}
\end{equation} 
which exists since $Y$ is cofibrant and $\varepsilon$ is a trivial fibration. It follows 
that~$ \widetilde{\mathsf{app}}(\widetilde{\lambda}(b) \times_A 1_A) = b$, 
so the $\beta$-rule holds as an equality. Instead, for $f \co X \to \mathbb{L}(\Pi_A(B))$, we have a homotopy
\[
\eta_f  \co \widetilde{\lambda}( \widetilde{\mathsf{app}} \, (f \times_A 1_A)) \sim  f  \, ,
\]
which is constructed as the diagonal filler in the following diagram
\[
\xymatrix@C=3cm{
\partial \Delta[1] \times Y \ar[r]^-{\big[f \, , \;  \widetilde{\lambda} ( 
\widetilde{\mathsf{app}} \,  (f \times 1_A) ) \big]} \ar[d] & \mathbb{L}(\Pi_A(B)) \ar[d]^\varepsilon \\
\Delta[1] \times Y \ar[r] \ar@{.>}[ur] & \Pi_A(B) \, , }
\]
where the bottom map is the constant homotopy given by the equality in the $\eta$-rule in~\eqref{equ:betaeta}. 

This result will allow us to equip the comprehension category
of cofibrant simplicial sets~\cref{thm:compcat} with weakly stable $\Pi$-types. We 
will take the $\Pi$-types in the comprehension category to be given by the objects
$ \mathbb{L}(\Pi_A(B))$. These will satisfy  stability properties that are weaker than those satisfied by $\Pi_A(B)$ in the classical simplicial model, 
essentially the same as those enjoyed by $\mathsf{Id}$-types in homotopy-theoretic models~\cite{awodey-warren:homotopy-idtype}, which are weaker even than those considered in~\cite{CarboniA:locccec}.  

Let us also note that the requirement of cofibrancy of $Y$ is used in an 
essential way to construct the diagonal filler in~\eqref{equ:diag-for-beta}. For this
reason, it seems impossible to obtain an analogous result for the comprehension category 
with base the category of all simplicial sets and total category that of Kan fibrations 
with cofibrant domain.
\end{remark}

\section{The weak equivalence extension property}
\label{sec:equep}

The main goal of this section is to prove the so-called weak equivalence extension property, which will be the key to prove the univalence of the classifying fibrations considered in~\cref{sec:unifbb,sec:fibrancy-and-univalence}.  For this, we follow closely the approach in \cite{voevodsky-simplicial-model}, but exploiting crucially the cofibrancy requirements that are part of our set-up. Another proof is given in~\cite[Section~3.2]{GambinoN:anocp}, adapting the argument in~\cite{SattlerC:equepu}.

\begin{lemma}\label{Lemma:ForTheExtProperty} Let  $f \co Y \rightarrow X$ be a cofibration between
cofibrant objects. 
\begin{enumerate}[$(i)$]
\item The dependent product along $f$, $\Pi_f \co \SSet_{/Y} \rightarrow \SSet_{/X}$, preserves trivial fibrations.
\item The counit of the adjunction $f^* \dashv \Pi_f$ is a natural isomorphism.
\item If $q \co B \to Y$ is a map with cofibrant domain, then $\Pi_f(q) \co \Pi_Y(B) \to X$  is so.
\item Trivial fibrations extend along $f$, \ie given a trivial fibration $q
 \co B \to Y$  as in the solid diagram
\[
\xymatrix{
B \ar@{.>}[r]^{g} \ar[d]_{q} \drpullback  & A \ar@{.>}[d]^{p} \\
Y \ar[r]_-f &  X \, ,}
\]
then there exists a trivial fibration $p \co A \rightarrow X$ which fits in the dotted pullback square above. Moreover if $B$ is cofibrant  then 
$A$ can be chosen to be 
cofibrant as well.
\end{enumerate}
\end{lemma}

\begin{proof} We prove the different parts separately. 

For part~(i), since the dependent product functor $\Pi_f$ is right adjoint to the pullback functor~$f^*$ and trivial fibrations are the maps with the the right lifting property with respect to cofibrations, $\Pi_f$ preserves trivial fibrations if and only if $f^*$ preserves cofibrations. But this follows from~\cref{thm:cof-pbk}.

For part~(ii),  since $f$ is a monomorphism,  $\Sigma_f \co \SSet_{/X} \rightarrow \SSet_{/Y}$ is fully faithful and hence the unit $\eta \co 1 \rightarrow f^*  \Sigma_{f}$ is an isomorphism. By adjointness, the counit $\varepsilon \co f^* \Pi_f \rightarrow 1$ is also an isomorphism. 

For part (iii), let $y \co  \Delta[n] \rightarrow \Pi_Y(B)$ be a  $k$-cell. We will show that for a given degeneracy $\sigma \co [n] \rightarrow [k]$ it is decidable if $y$ is ``$\sigma$-degenerated'', i.e. if $y$ factors through $\sigma \co \Delta[n] \rightarrow \Delta[k]$. As degeneracy is decidable in $X$, one can freely assume that the image of   $y$ in $X$ via $\Pi_f(q)$ is $\sigma$-degenerated (as if it is not the case, $y$ is not $\sigma$-degenerated) hence one has a solid diagram:

\[
\xymatrix{
\Delta[n] \ar[d]_{\sigma} \ar[r] & \Pi_Y(B) \ar[d] \\
\Delta[k] \ar@{.>}[ur]^? \ar[r] & X \, . \\
}
\]
Because of the adjunction between dependent product and the pullback along $f$, the existence of a lift as above is equivalent to the existence of a lift in:
\[
\xymatrix{
f^* \Delta[n] \ar[d]_{f^* \sigma} \ar[r]^x & B \ar[d] \\
f^* \Delta[k] \ar@{.>}[ur]^? \ar[r] & Y \, .  \\
}
\]
The objects $f^* \Delta[n]$ and $f^* \Delta[k]$ are decidable simplicial subset of $\Delta[n]$ and $\Delta[k]$ because $f$ is itself a levelwise complemented monomorphism, hence they are both finite decidable simplicial sets. The map $f^* \sigma$ is degeneracy quotient by \cref{prop:Degen_quotient_are_pullback_stable}), hence \cref{lem:decidability_lift_degen_quo} shows that the existence of such a lift is decidable (as degeneracy quotients are epimorphisms, the existence of lift making the upper triangle commutes is equivalent to the existence of lift making the square commutes).

For part~(iv), given a trivial fibration $q \co B \to Y$, define $p \co A \to X$ to  be $\Pi_f(q) \co \Pi_Y(B) \to Y$. This map is a trivial fibration by part~(i) and the square is a pullback by part (ii). The final remark about the cofibrancy of $A$ follows from part~(iii). \qedhere
\end{proof}

\begin{proposition}[Weak equivalence extension property]
\label{Prop:Homotopy_ext_prop}
Let 
\[
\xymatrix{
B \ar[r]^g \ar@{->>}[d]_q & A \ar@{->>}[d]^p \\
Y \ar[r]_f & X \, ,}
\]
be a commutative diagram where $p \co A \to X$ and $q \co B \to Y$ are fibrations with cofibrant domains and $f \co Y \to X$ is
a cofibration between cofibrant objects. Assume that the map $u \co B \to A[f]$ defined by~$u \defeq (q, g)$, fitting the diagram 
of solid maps
\[ 
\xymatrix{
 B
  \ar@{.>}[rr]
  \ar[dr]^{u}
  \ar@{->>}[dd]_(.3){q}
&&
  \bar{B}
  \ar@{.>}[dr]^{v}
  \ar@{.>>}[dd]_(.3){\bar{q}}|{\hole}
&\\&
  A[f] 
  \ar[rr]
  \ar@{->>}[dl]
&&
  A
  \ar@{->>}[dl]^{p}
\\
  Y
  \ar[rr]_{f}
&&
  X \, ,
&
}
\]
is a weak equivalence in $\SSet_{/ Y}$. Then there exist a fibration $\bar{q} \co \bar{B} \to X$, a weak equivalence $v \co \bar{B} \to A$ in $\SSet_{/X}$ and a map $B \to \bar{B}$ such that both squares in the diagram above are pullbacks. 
\end{proposition}

\begin{proof} We define the required object $\bar{B}$ as the following pullback:
\[\xymatrix{
\bar{B} \ar[d] \ar[r] \drpullback & \Pi_Y(B) \ar[d] \\
A \ar[r]_-{\eta_{A}} & \Pi_Y \big( A[f]  \big) \, ,
}\]
where $\eta_{A}$ is  a component of the unit of adjunction  $f^* \dashv \Pi_f$. An application of the pullback $f^* \co \SSet_{/X} \to \SSet_{/Y}$ to this square gives the commutative square
\[\xymatrix{
\bar{B}[f] \ar[d] \ar[r]  & B \ar[d] \\
A[f] \ar@{=}[r] &A[f] \, ,
}\]
which is a pullback since $f^* \Pi_f \iso 1$ by part~(ii) \cref{Lemma:ForTheExtProperty}. Hence 
$B \iso \bar{B}[f]$, as required.

Since $B$ is cofibrant, we have that $\Pi_Y(B)$ is cofibrant by part~(iii) of  \cref{Lemma:ForTheExtProperty}. Hence, the simplicial set~$\bar{B}$  is also cofibrant by \cref{lem:cofibrant_fiber_product}. Furthermore, the maps $B \rightarrow \bar{B}$ and~$A[f] \rightarrow A$ are cofibrations by~\cref{thm:cof-pbk}, as they are pullback of the cofibration~$f \co Y \rightarrow X$.

It remains to prove that $v \co \bar{B} \rightarrow A$ is a weak equivalence and that $\bar{q} \co \bar{B} \rightarrow X$ is a fibration. The map $u$ can be factored into a trivial cofibration followed by a trivial fibration, and it is sufficient to prove these claims for each half of the factorization separately, \ie when~$u$ is a trivial fibration and when it is a trivial cofibration.

If $u$ is a trivial fibration, then its image under $\Pi_f$ is a trivial fibration by 
part~(i) of \cref{Lemma:ForTheExtProperty}. Since the map $\bar{B} \rightarrow A$ is a pullback of this map,
it is also a trivial fibration. This also implies that the composite $\bar{B} \rightarrow A \rightarrow X$ is a fibration.

We now assume that $u \co B \rightarrow A[f]$ is a trivial cofibration. Using that the maps from $\bar{B}$ and $A[f]$ to $Y$ are fibrations, we can show that $u$ is a strong deformation retract over $Y$, \ie there is a retraction $r \co A[f] \rightarrow B$ of $u$ in $\SSet_{/Y}$ and a homotopy 
\[
H \co \Delta[1] \times A[f] \rightarrow A[f]
\] 
between 
$u \circ r$ and $1_{A[f]}$, whose composite with $A[f] \rightarrow Y$ is the trivial homotopy. Indeed, $r$ and  $H$ are respectively constructed as the dotted diagonal liftings in the squares:

\[\xymatrix{
B \ar[d] \ar@{=}[r]  & B \ar@{->>}[d] \\
A[f] \ar[r] \ar@{.>}[ur]^r  & Y \, , 
} \quad \xymatrix{
\displaystyle A[f] +_B \left( B \times \Delta[1] \right) +_B A[f] \ar[d] \ar[r]  & A[f] \ar@{->>}[d] \\
A[f] \times \Delta[1] \ar[r] \ar@{.>}[ur]_H  & Y \, , 
}\]

We want to show that $\bar{B} \rightarrow A$ is also a deformation retract by constructing a similar homotopy 
\[
H' \co \Delta[1] \times A \rightarrow A \, .
\] 
This homotopy will be constructed so that it is $H$ on $\Delta[1] \times A[f]$ ,  it is the map 
\[
\Delta[1] \times \bar{B} \rightarrow \Delta[0]  \times \bar{B} \iso \bar{B} \rightarrow A
\] 
on $\Delta[0] \times \bar{B} $ (indeed they agree on $\Delta[1] \times B$) and it is the identity on $\Delta[0] \times A$.  This is achieved by taking a diagonal filling in the square:
\[
\xymatrix@C=1.5cm{
\big( \Delta[1] \times (\bar{B} \cup A[f]) \big)  \cup \big( \Delta[0] \times A \big) \ar@{>->}[d] \ar[r] & A \ar@{->>}[d] \\
\Delta[1] \times A \ar[r] \ar@{.>}[ur]^{H'} & X \, .
}\]
Such a diagonall filler exists since the map on the left-hand side is a trivial cofibration, being the 
 pushout-product of $Y_0 \co \Delta[0] \rightarrow \Delta[1]$ and the cofibration $\bar{B} \cup A[f] \rightarrow A$, and the map on the right-hand side is a fibration by assumption.

It remains to see that the map $H_{1} \co A \rightarrow A$ is indeed a retraction of $\bar{B} \rightarrow A$. We already know that the restriction of $H_{1}$ to $\bar{B}$ is  the inclusion of $\bar{B}$ in $A$, so it is enough to show that $H_{1}$ has values in $\bar{B}$. We also know that $H_{1}$ restricted to $A[f]$ takes values in $B \subseteq \bar{B}$. By definition of $\bar{B}$, the map $H_1$ factors into $\bar{B}$ if and only if it takes values in $\Pi_Y(B)$ when seen as a map to $\Pi_Y(A[f])$, and by adjunction this is the case if and only if the map corresponding to $H_1$, $A[f]= f^*(A) \rightarrow A[f]$ takes values in $B$, but this is indeed the case, as already mentioned above.

Since $\bar{B} \rightarrow A$ is a deformation retract, it is invertible in the homotopy category and so it is a weak equivalence. The construction above also shows that~$\bar{B}$ is retract of $A$ in $\SSet_{/X}$ and hence $\bar{q} \co \bar{B} \rightarrow X$ is a fibration because $p \co A \rightarrow X$ is.
\end{proof}

 Another proof of~\cref{cor:Pi_types_are_fibrant} is given in~\cite{GambinoN:anocp} modifying appropriately
the arguments in~\cite{SattlerC:equepu}.

\section{A  classifying small fibration between bifibrant objects}
\label{sec:unifbb}

The aim of this section is to begin establishing the results necessary to have a pseudo Tarski
universe satisfying the univalence axiom in the comprehension category of cofibrant simplicial
sets. For this, we need to define a cofibrant Kan complex $\U_c$ and a Kan 
fibration with cofibrant domain $\pi_c \co \UU_c \to \U_c$. The closure of the type-theoretic universe under various type-formation operations
corresponds to the closure of the class of maps arising as pullbacks of $\pi_c$ under the 
operations necessary to intepret the corresponding types. For this, it is convenient to 
consider~$\pi_c$ to be a fibration that classifies small fibrations between cofibrant simplicial
sets, \ie such that for every such fibration $p \co A \to X$ there exists a map $a \co X \to \U_c$ fitting in a pullback diagram of the form
\[
\xymatrix{
A \drpullback \ar[r] \ar@{->>}[d]_p   & \UU_c \ar@{->>}[d]^{\pi_c} \\
X \ar[r]_-a &  \U_c \, .}
\]
Here, note that we make no requirement for
the map $a$ to be unique, in contrast for example with the situation of the subobject classifier
in an elementary topos. Indeed, the map $a$ is not unique, but only unique  up to a contractible space
of choices, a fact that will be expressed by showing that $\pi_c$ is univalent. 

In this section, we limit ourselves to define $\pi_c \co \UU_c \to \U_c$ and prove that it
classifies small fibrations between cofibrant simplicial sets. We will then show 
that $\U_c$ is bifibrant, that $\UU_c$ is cofibrant  and that $\pi_c$ is univalent 
in~\cref{sec:fibrancy-and-univalence}. For the goal of this section, we proceed in two steps. First, we modify  the construction of the weak classifier for small fibrations in~\cite{voevodsky-simplicial-model} to obtain a small fibration $\pi \co \UU \to \U$ which classifies small fibrations with cofibrant fibers. Since the base of this fibration does not appear to be cofibrant,
we then consider a suitable cofibrant replacement  of~$\U$ and obtain the required fibration $\pi_c \co \UU_c \to \U_c$ via a pullback. 

\medskip

As a preliminary step, let us recall that a simplicial set $A$ is \emph{small}  if $A_n$ is a small set for every $[n] \in \Delta$ and that a map $p \co A \to X$ of simplicial sets is \emph{small} if for every $x \co \Delta[n] 
\to X$ the simplicial set~$A[x]$ given by the pullback square
\[
\xymatrix{
A[x] \ar[r] \ar[d] \drpullback & A \ar[d]^{p} \\
\Delta[n] \ar[r]_-{x} & X \, , }
\]
is small. Let us also recall the  construction of a map of simplicial sets $\rho \co \VV \to \V$ that classifies
small maps of simplicial sets, which is a special case of the results in~\cite{hofmann-streicher-universes} for arbitrary presheaf categories.  For this, we use the equivalence in~\eqref{equ:psh-slice-sset} and the notation associated to it.
The simplicial set $\V$ is defined by letting
\[
\mathsf{V}_n \defeq \{ F \co {\Delta_{/[n]}}^{\op} \to \Set \ | \ \pi_F \co \textstyle{\int F} \to \Delta[n] \text{ is a small
map} \}
\]
for $[n] \in \Delta$. The object $\VV$ and the map $\rho \co \VV \to \V$ are then defined in an evident way.

\bigskip

We now come to our first step, in which we define a small fibration $\pi \co \UU \to \U$ which 
classifies the class of small fibrations with cofibrant fibers. 

\begin{definition} We say that a map $p \co A \to X$ \emph{has cofibrant fibers} if for every $x \co \Delta[n] 
\to X$ the simplicial set~$A[x]$ given by the pullback square
\[
\xymatrix{
A[x] \ar[r] \ar[d] \drpullback & A \ar[d]^{p} \\
\Delta[n] \ar[r]_-{x} & X \, , }
\]
is cofibrant.
\end{definition}

\begin{lemma} \label{lem:Cofib_fiber} \hfill 
\begin{enumerate}[$(i)$]
\item If a map $p \co A \rightarrow X$ has cofibrant domain then it has cofibrant fibers.
\item If $X$ is cofibrant and $p \co A \rightarrow X$ has cofibrant fibers then $A$ is cofibrant.
\end{enumerate}
\end{lemma} 

\begin{proof} Part~(i) follows from \cref{lem:cofibrant_fiber_product}. For part~(ii), let $[n] \in \Delta$, $a \in A_n$. Since $X$ is cofibrant, by the constructive version of the Eilenberg-Zilber lemma we can write $p(a) \in X$ in a unique way as $p(a) =s^*(x)$, where $s \co [n] \to [k]$ is a degeneracy and 
$x \in X_k$ is a non-degenerate cell. Let $x \co \Delta[k] \rightarrow X$ be the corresponding map. We now form the pullback
\[
\xymatrix{
A[x] \ar[r]^w \ar[d] \drpullback & A  \ar[d]^{p}  \\
\Delta[k] \ar[r]_{x} & X \, . }
\]
By the universal property of the pullback, there is a unique cell $e \in A[x]_n$ such that $w(e)=a$, and the image of $e$ in $\Delta[k]$ is the cell $s \co [n] \to [k]$, whose image in $X$ are both equal to $p(a)=s^*(x)$.

By the assumption that $p$ has cofibrant fibers, the simplicial set $A[x]$ is cofibrant and hence it is decidable whether $e$ is degenerate or not. We claim that $a$ is degenerate if and only if $e$ is, which implies that it is decidable whether $a$ is degenerate.

 Indeed as $a = w(e)$ then if $e$ is degenerate so $a$ is. Conversely, assume that $a=\sigma^*(y_1)$ for a non-trivial degeneracy $\sigma$. Then $p(a)=\sigma^*(x_1)$, hence by the uniqueness part of the Eilenberg-Zilber lemma for $X$ one has that $s=s_1 \circ \sigma$ for some degeneracy $s_1$, and $x_1 = s_1^*(x')$. In particular, we get a unique cell $e_1$ of $A[x]$ whose image in $\Delta[n]$ and $X$ are  $s_1$ and $a_1$, respectively, whose images in $X$ are both equal to $x_1=s_1^*(x')$. Finally, the image of $p^*(e_1)$ in $\Delta[n]$ and~$A$ are  $p^* y_1 =a$ and $s_1 \circ p =s$, respectively, and hence $p^*(e_1) =e$, which proves that $e$ is degenerate as soon as $a$ is.
\end{proof}

Define an object $\U$ by letting, for $[n] \in \Delta$, 
\[
\U_n = \{ F \in \V_n \ | \ \pi_F \co \textstyle{\int F} \to \Delta[n] \text{ is a small fibration and 
$ \textstyle{\int F}$ is cofibrant} \} \, .
\]
Here, recall that by a fibration we  mean a function together with a fibration structure, \ie a choice of lifts against the horn inclusions $\Lambda^k[n] \to \Delta[n]$.
There is then an evident forgetful map $\U \to \V$, which is not injective. Here, $\U$ is a presheaf because $\V$ is a presheaf and the pullback of a fibration (with chosen lifts) is again a fibration (with chosen lifts), using \cref{lem:cofibrant_fiber_product} to obtain the required cofibrancy conditions.

We then define the map $\pi \co \UU \to \U$ via the pullback 
\begin{equation}
\label{equ:def-of-pi}
\begin{gathered}
\xymatrix{
\UU \ar[r] \ar[d]_\pi \drpullback  & \VV \ar[d]^\rho \\
\U \ar[r] & \V \, . }
\end{gathered}
\end{equation}

\begin{proposition} \label{thm:universe-u}  \hfill 
\begin{enumerate}[(i)] 
\item The map $\pi \co \UU \to \U$ is a small fibration with cofibrant fibers.
\item The map $\pi \co \UU \to \U$ classifies small fibrations with cofibrant fibers. More precisely, a map $v \co X \to \V$ can be lifted to a map $X \to \U$ if and only if it classifies a fibration with cofibrant fibers. Moreover, choice of a lift to $\U$ corresponds (functorially in $X$ and $V$) to the choice of a fibration structure on the map classified by $v$.

\end{enumerate}
\end{proposition}

\begin{proof} We prove the two claims separately.
For part~(i), consider a map $a \co \Delta[n] \rightarrow \U$ and the pullbacks
\[
\xymatrix{
A \ar[r] \ar[d]_p  \drpullback & \UU \ar[d]^{\pi} \ar[r]  \drpullback & \VV \ar[d]^{\rho} \\
\Delta[n] \ar[r]_{a} & \U \ar[r] & \V \, . }
\]
This shows that $p \co A \rightarrow \Delta[n]$ is isomorphic to $\pi_F \co \int F \rightarrow \Delta[n]$ in $\SSet_{/ \Delta[n]}$, where $F$ corresponds under the equivalence in~\eqref{equ:psh-slice-sset} to 
$a \co \Delta[n] \rightarrow \V$. Therefore, by definition of $\U$, $A$ is cofibrant and $p \co A \rightarrow \Delta[n]$ is a small fibration. This implies that $\pi \co \UU \rightarrow \U$ has cofibrant fibers. To show it is a fibration, we rewrite a general lifting problem against a horn inclusion $h^n_k \co \Lambda^k[n] \rightarrow \Delta[n]$ as follows:
\[
\xymatrix{
\Lambda^k[n] \ar[r] \ar[d]_{h^k_n} & A  \ar[d]^{p} \ar[r] & \UU \ar[d]^\pi \\
\Delta[n] \ar@{=}[r]  & \Delta[n] \ar[r]_a & \U \, ,  }
\]
and then use that $p \co A  \to \Delta[n]$ is a fibration, with chosen lifts, which hence produce a chosen lift in the diagram above.

For part~(ii), if a map $v \co X \to \V$ lifts to $\U$, then the map $A \to X$ classified by $v$ is a pullback of $\UU \to \U$ so it is a small fibration with cofibrant fibers. Conversely, let $p \co A \to X$ be a small fibration with cofibrant fibers. Being a small 
map, $p$ fits into a pullback of the form 
\[
\xymatrix{
A \ar[r] \ar[d]_p \drpullback  &  \VV \ar[d]^{\rho} \\
X \ar[r]_{a} & \V \, . }
\]

For each cell of $X$,  $x \co \Delta[n] \rightarrow X$ the pullback $A[x] \rightarrow \Delta[n]$ is a small fibration with cofibrant domain by the assumptions on $p$, which means, by the definition of $\U$ that the image of the cell $x$ in $\V$ lifts to $\U$. As one can easily check that this lifts to $\U$ are functorial in $[n]$, the map $X \rightarrow \V$ lifts to $\U$, and hence $A \rightarrow X$ is actually the pullback of $\UU \rightarrow \U$. One also easily check that there is a correspondence between fibration structures on the map $A \to X$ and compatible choices of fibration structures on each pullback along a map $\Delta[n] \to X$, so that the two construction explained above are inverse of each other. \qedhere
\end{proof}

\bigskip

We now  construct a fibration $\pi_c \co \UU_c \to \U_c$ that classifies small fibrations between 
cofibrant objects. In particular, the base $\U_c$ of this fibration will be cofibrant.  In order to do this, let $\U_c$ be a cofibrant replacement of~$\U$,  which comes with a trivial fibration
\begin{equation}
\label{equ:ucu}
\tau \co \U_c \rightarrow \U \, .
\end{equation}
We then define $\UU_c$ via the pullback
\begin{equation}
\label{equ:def-of-pi-c}
\begin{gathered}
\xymatrix{
\UU_c \ar[d]_{\pi_c} \ar[r] \drpullback & \UU \ar[d]^{\pi}  \\
\U_c \ar[r]_\tau & \U \, .}
\end{gathered}
\end{equation}
We now prove that $\pi_c \co \UU_c \to \U_c$ has the required properties.

\begin{theorem} \label{thm:universe-uc} 
\hfill 
\begin{enumerate}[(i)] 
\item $\pi_c \co \UU_c \to \U_c$ is a small fibration with cofibrant fibers and cofibrant codomain. 
\item The map $\pi_c \co \UU_c \to \U_c$ classifies small fibrations between cofibrant
objects, \ie if a map $p \co A \to X$, with $X$ cofibrant, is a small fibration between cofibrant objects
if and only there exists a pullback diagram of the form
\[
\xymatrix{
A \ar[r] \ar@{->>}[d]_p \drpullback & \UU_c \ar@{->>}[d]^{\pi_c} \\
X \ar[r]_a & \U_c \, . }
\]
\end{enumerate}
\end{theorem}

\begin{proof} For part~(i), $\U_c$ is cofibrant by construction and the rest of the claim follows from part~(ii) of~\cref{thm:universe-u}. 


For part~(ii), let $p \co A \to X$ be a small fibration between cofibrant objects, by part~(i) of \cref{lem:Cofib_fiber}, it has cofibrant fibers and hence by part~(ii) of \cref{thm:universe-u} there is a pullback diagram of the form 
\[
\xymatrix{
A \ar[r] \ar@{->>}[d]_p \drpullback & \UU \ar@{->>}[d]^{\pi} \\
X \ar[r]_a & \U \, . }
\]
Since $X$ is cofibrant, we have a diagonal filler in the diagram
\[
\xymatrix{
0 \ar[r] \ar[d] & \U_c \ar[d]^{\tau} \\
X \ar[r]_a \ar@{.>}[ur]^{a_c} & \U \, . }
\]
 We then obtain the diagram
\[
\xymatrix{
A \ar[r] \ar@{->>}[d]_p &  \UU_c \ar[r]  \ar@{->>}[d]^{\pi_c}  & \UU \ar@{->>}[d]^{\pi} \\
X \ar[r]_{a_c} & \U_c \ar[r]_{a} &  \U \, . }
\]
Here, the right-hand side square and the rectangle are pullbacks and therefore the left-hand
side square is also a pullback. Hence $p$ is indeed a pullback of $\UU_c \rightarrow \U_c$. Conversely, any pullback  $p \co A \rightarrow X$ of $\UU_c \rightarrow \U_c$ is a small fibrations with cofibrant fibers, hence by part~(ii) of \ref{lem:Cofib_fiber}, if $X$ is cofibrant, $p$ is small fibration between cofibrant objects.
\end{proof}

\section{Fibrancy and univalence of the universe}
\label{sec:fibrancy-and-univalence}

The aim of this section is to show that $\U_c$ is a cofibrant Kan complex and that $\pi_c \co
\UU_c \to \U_c$ is univalent.  For this, let us first return to consider the fibration $\pi \co \UU \to \U$ defined
via the pullback in~\eqref{equ:def-of-pi}. Let $\U^{\rightarrow}$ be the simplicial set whose $n$-simplices are triples of the form $(F_0, F_1, \phi)$, where $F_0$ and $F_1$ are $n$-simplices of $\U$, \ie functors
 \[
F_0, F_1 \co {\Delta_{/[n]}}^{\op} \rightarrow \Set
\]
and $\phi \co F_0 \Rightarrow F_1$ is a natural transformation. By the equivalence in~\eqref{equ:psh-slice-sset},
such triples correspond to commutative diagrams of the form
\begin{equation}
\label{equ:corresp}
\begin{gathered}
\xymatrix{ 
A_1 \ar[rr]^{f} \ar@{->>}[dr]_{p_1} & & A_2 \ar@{->>}[dl]^{p_2} \\
& \Delta[n] \, , & }
\end{gathered}
\end{equation}
where $p_1$ and $p_2$ are fibrations with cofibrant domain.

\begin{lemma} The map $(\mathsf{dom}, \mathsf{cod}) \co \U^{\rightarrow} \rightarrow \U \times \U$ is a fibration
\end{lemma}

\begin{proof} Observe that $\U^{\rightarrow}$ is exactly $\Pi_p(\UU \times \UU)$, 
where $p \co \UU \times \U \rightarrow \U \times \U$ is the evident map. 
It follows from \cref{cor:Pi_types_are_fibrant} that $\U^{\rightarrow} \rightarrow \U \times \U$ is a fibration.
More precisely, \cref{cor:Pi_types_are_fibrant} implies that any pullback of  $\U^{\rightarrow} \rightarrow \U \times \U$ to a cofibrant
$X \rightarrow \U \times \U$ is a fibration (due to the cofibrancy assumption of  \cref{cor:Pi_types_are_fibrant}),
but this is sufficient to prove that $\U^{\rightarrow} \rightarrow \U \times \U$ is a fibration,
as in the argument for part~(i) of  \cref{thm:universe-u}.
\end{proof}

We recall that $\BFFib_{/X}$ denotes the full subcategory of $\SSet_{/X}$ of fibrations with cofibrant domain. Given $f \co A \to B$ an arrow in $\BFFib_{/X}$ we now construct the object $\Iseq(f)$ in $\BFFib_{/X}$ of weak equivalences structures on $f$. We will use the 
characterisation of weak equivalences as maps that have both a left inverse and a right inverse, \ie what are called bi-invertible maps in \cite[Section 4.3]{hottbook}. 
For this, we essentially follow \cite{ShulmanM:unifid} and \cite[Section 1.4]{StenzelR:unirca}, which in turn are a category-theoretic counterpart of the type-theoretic
treatment in \cite{hottbook}. We cannot use the results in \cite{ShulmanM:unifid} and \cite[Section 1.4]{StenzelR:unirca} directly since they both work under 
assumptions on path objects that would not be satisfied for us without restricting to cofibrant objects and assume that $\Pi$-types are given by dependent products, i.e. right adjoints to pullback functors, which we do not have if we restrict to cofibrant objects. For these reasons, we need to prove \cref{thm:black-box} explicitly. However, since this is very similar to the references cited above, we will not give all details.

Given a map $f:A \to B$ in $\BFFib_{/X}$, a right inverse for $f$, is a map $h:B \to A$ in $\BFFib_{/X}$ and a homotopy $f \circ h \sim \Id_B$, given by an arrow $B \to \Path_{/X}(B)$ in $\BFFib_{/X}$. Similarly, a left inverse for $f$ is an arrow $g:B \to A$ together with a homotopy $g \circ f \sim \Id_A$, given by an arrow $A \to \Path_{/X}(A)$ in $\BFFib_{/X}$.

Using dependent products (without taking cofibrant replacements), we can construct fibrant objects $\Linv(f) \twoheadrightarrow X$ and $\Rinv(f) \twoheadrightarrow X$, that are respectively the objects of left and right inverse of $f$. Explicitely, for an object $Y \in \SSet_{/X}$ a map $Y \to \Linv(f)$ is the same as a left inverse to the map $A \times_{X} Y \to B \times_X Y$ in $\BFFib_{/Y}$.

We then define $\Iseq(f) = \Linv(f) \times_X \Rinv(f)$. It is also fibrant over $X$, and an arrow $Y \to \Iseq(f)$ in $\SSet_{/X}$ is the same as the data of both a left inverse and a right inverse of the pullback of $f$ to $\SSet_{/Y}$.

\begin{proposition} \label{thm:black-box}
Let $f:U \to V$ be an arrow in $\BFFib_{/X}$ between two fibrations with cofibrant fibers $U \twoheadrightarrow X$ and $V \twoheadrightarrow X$. Let $A \hookrightarrow B$ a cofibration between cofibrant objects in $\SSet_{/X}$, then a lifting problem

\[
\xymatrix{ A \ar[r] \ar[d]^i & \Iseq(f) \ar[d] \\
B \ar[r]_{p} & X
}
\]
has a diagonal filling if and only if the pullback of $f$ to $\SSet_{/B}$ is a weak equivalence.
\end{proposition}

\begin{proof} Let  $p^* \co \SSet_{/X} \to \SSet_{/B}$ be the pullback along $p$.
 By construction, if we have a lift $B \to \Iseq(f)$ then we get both a left and a right inverse to $p^*(f) \co p^*(U) \to p^*(V)$, which is an arrow between bifibrant objects in $\SSet_{/B}$. This implies that $p^*(f)$ is a weak equivalence, and concludes the proof of one implication.

Conversely, we assume that $p^*(f): p^*(U) \to p^*(V)$ is a weak equivalence. Note that as~$U$ and~$V$ are assumed to have cofibrant fibers over $X$ and $A$ and $B$ are cofibrant, all the objects $p^*U$, $i^*U$, $p^*V$ and $i^* V$ are cofibrant.

Because $\Iseq(f)$ is constructed as $\Rinv(f) \times_X \Linv(f)$, we only need to construct lifts to $\Linv(f)$ and $\Rinv(f)$ separately. We start with the lift to $\Rinv(f)$. A right inverse to $f : U \to V$ in $\SSet_{/X}$ can be described as a section to a fibration $\pi_f: U \times_V \Path_{/X}(V) \to V$. Indeed, in terms of generalised elements it should associate to each $v \in V$ an element $u\in U$ together with a homotopy between $f(u)$ and $v$. The construction  $f \mapsto \pi_f$ is stable under pullback in the sense that $\pi_{p^* f} \cong p^*(\pi_f)$. Morever, if $f$ is an equivalence $\pi_f$ also is, and hence it is a trivial fibration, in particular $p^*(\pi_f)$ and $i^*(\pi_f)$ are trivial fibrations. Now construing a lift to $\Rinv(f)$ is the same as constructing a section of $p^*(\pi_f)$ to $B$ (which is a trivial fibration) that extends the given section of $i^*(\pi_f)$. By our cofibrancy assumption, and \cref{thm:cof-pbk}, the horizontal maps in the square:
\[
\xymatrix{  i^*\left( U \times_V \Path_{/X}(V) \right) \ar[r] \ar[d]^{i^* \pi_f} & p^*\left( U \times_V \Path_{/X}(V) \right)  \ar[d]^{p^*(\pi_f)} \\
i^*V \ar[r] & p^* V
}
\]
are cofibrations, and so the desired section on $p^*\pi_f$ extending the one on $i^*\pi_f$ can be constructed using the right lifting property of the trivial fibration $p^* \pi_f \simeq \pi_{p^* f}$ against $p^*V \to i^* V$ .

We now move to the lift to $\Linv(f)$. A left inverse of $f$, can be described dually as a retraction of the map $\sigma_f: U \hookrightarrow \Delta[1] \times U \coprod_U V$ in $\SSet_X$. Similarly to the previous case, the construction $f \mapsto \sigma_f$ is also stable under pullback (i.e. $p^*\sigma_f \cong \sigma_{p^* f}$), the map $\sigma_f$ is cofibration when $U$ and $V$ are cofibrant and is a trivial cofibration if in addition $f$ is a weak equivalence. So, as before, $p^*(\sigma_f)$ and $i^*(\sigma_f)$ are trivial cofibrations. Thus, the lift that we are trying to construct is equivalent to a retraction in $\SSet_{/B}$ of the right map in the square below that extends the retraction already existing in $\SSet_{/A}$ of the left map:

\[
\xymatrix{
i^* U \ar[r] \ar[d]^{i^* \sigma_f} & p^* U   \ar[d]^{p^*(\sigma_f)} \\
i^*\left( \Delta[1] \times U \coprod_U V \right) \ar[r] & p^* \left( \Delta[1] \times U \coprod_U V \right) \, .
}
\]

This can be achieved by chosing a lift in:

\[
\xymatrix{  p^* U \coprod_{i^* U} i^*\left( \Delta[1] \times U \coprod_U V \right) \ar[r] \ar[d] & p^* U   \ar[d] \\
p^* \left( \Delta[1] \times U \coprod_U V \right) \ar[r] \ar@{.>}[ur] & B
}
\]
Here, the map on the right is a fibration. The left arrow is a weak equivalence by $2$-out-of-$3$, since both maps $p^*(\sigma_f)$ and (the pushout of) $i^* \sigma_f$ are trivial cofibration. It is also a cofibration since it is a monomorphism and the decidability condition hold because all the maps in the previous square are cofibrations.
\end{proof}

Finally, we define $\Weq(\U) \twoheadrightarrow \U^\to$ by applying this construction $\Iseq(f)$, to the universal arrow between small fibrations with cofibrant fibers, which lives over $\U^\to$. By \cref{thm:black-box}, we have the following result.

\begin{proposition} \label{prop:Weq_classify_Weq}
The map $\Weq(\U) \to \U^\to$ is a fibration and given $A \rightarrowtail B$ a cofibration between cofibrant objects, a lifting problem
\[
\xymatrix{ A \ar[r] \ar@{>->}[d] & \mathsf{Weq}(\U) \ar[d] \\
B \ar@{.>}[ur] \ar[r]_{p} & \U^\to
}
\]
has a dotted diagonal filling if and only if the map classified by $p$ is a weak equivalence. \qed
\end{proposition}




\begin{theorem} \label{thm:fibrancy-of-u-and-uc} \hfill 
\begin{enumerate}[(i)] 
\item The simplicial set $\U$ is fibrant. 
\item The simplicial set $\U_c$ is bifibrant. 
\end{enumerate}
\end{theorem}

\begin{proof} We prove part~(i). Since
$(s, t) \co \mathsf{Weq}(\U) \rightarrow \U \times \U$ is a fibration, for any cofibrant 
simplicial set $X$,  maps $a_1 \, , a_2 \co X \rightrightarrows \U$ and homotopy $h \co \Delta[1] \times X \rightarrow \U$ from $a_1$ to $a_2$, there is a weak equivalence in $\SSet_{/X}$ between the objects classified by $a_1$ and $a_2$, constructed as follows. For this, we first consider a diagonal filler in the
diagram
\[
\xymatrix{ X \ar[r]^{i_1} \ar@{>->}[d]_{\delta^0} & \mathsf{Weq}(\U) \ar@{->>}[d] \\
\Delta[1] \times X \ar[r]_{(a_1,h)} \ar@{.>}[ur] & \U \times \U \, .
}
\]
Here, $i_1$ denotes a map classifying the identify of the object classified by $a_1$. By $a_1$ in the first component of the bottom arrow we mean the composite $\Delta[1] \times X \rightarrow X \rightarrow \U$. Composing the dotted arrow with $\delta^1$ gives us a map $X \rightarrow  \mathsf{Weq}(\U)$ whose projection to~$\U \times \U$ if $(a_1,a_2)$, \ie it classifies a weak equivalence between the objects classified by $a_1$ and $a_2$. One can do the same thing with $\delta^0$ and $\delta^1$ exchanged to get a weak equivalence in the other direction.

Using this fact, we can now prove that $\U$ is fibrant. A map $h^k_n \co \Lambda^k[n] \rightarrow \U $ classifies a fibration $q \co B \rightarrow \Lambda^k[n]$ with cofibrant domain. The horn inclusions $h^k_n \co \Lambda^k [n] \rightarrow \Delta[n]$ fits into retract diagrams:
\[
\xymatrix{\Lambda^k[n] \ar@{>->}[d] \ar@{>->}[r] & \big( \Delta[1] \times \Lambda^k[n] \big) \cup \Delta[n] \ar@{>->}[d] \ar[r] & \Lambda^k[n] \ar@{>->}[d]  \\
\Delta[n] \ar@{>->}[r] & \Delta[1] \times \Delta[n] \ar[r] & \Delta[n] \, .
}\]

Where the map $\Delta[0] \rightarrow \Delta[1]$ can be either $\delta^0$ or $\delta^1$ depending on whether $0<k$ or $k<n$. See for example the last part of the proof of theorem 3.2.3 in \cite{joyal-tierney:simplicial-homotopy-theory}.

By the observation above, the composite map $\big( \Delta[1] \times \Lambda^k[n]  \big) \cup \Delta[n] \rightarrow \Lambda^k[n] \to \U$  gives a solid diagram of the form
\[ 
\xymatrix{
  B
  \ar@{.>}[rr]
  \ar[dr]
  \ar@{->>}[dd]
&&
  \bar{B}
  \ar@{.>}[dr]
  \ar@{.>>}[dd]|{\hole}
&\\&
  A'
  \ar[rr]
  \ar@{->>}[dl]
&&
  A
  \ar@{->>}[dl]
\\
  \Lambda^k[n]
  \ar[rr]_{h^k_n}
&&
  \Delta[n]
&
}
\] 
So we can construct a fibration  $\bar{q} \co \bar{B} \to \Delta[n]$ with cofibrant domain whose pullback 
along~$h^k_n$ is isomorphic to $q \co B \to \Lambda^k_n$. The map $b \co \Delta[n] \rightarrow \U$ classifying $\bar{q}$ gives the lift we are looking for. More precisely, we can use $q  \co \bar{B} \to
\Delta[n]$ to construct  a map $b \co \Delta[n] \rightarrow \U$ which extends the one we started from and classifies 
an object isomorphic to $\bar{B}$.

Part~(ii) follows immediately from part~(i) since $\tau \co \U_c \to \U$ of \eqref{equ:ucu} is a trivial fibration.
\end{proof}

We now wish to define  what it means for a small fibration with cofibrant fibers, and in particular for a small fibration between cofibrant objects, to  be univalent. For this, we fix such a fibration $p \co A \to X$ and construct an object $\Weq(p) \to X \times X$ that represents weak
equivalences between fibers of $p$, in the sense that maps $Y \to \Weq(p)$ in $\SSet_{/X \times X}$ are in bijective correspondence with tuples $(x_1, x_2, w,e)$ consisting of two map $x_1 \, , x_2 \co Y \to X$, a map $w \co A[x_1]
\to A[x_2]$ in $\SSet_{/Y}$, and $e$ a weak equivalence structure on $w$. The required object $\Weq(p)$ can be constructed as the pullback
\[
\xymatrix{
\Weq(p) \drpullback \ar[r] \ar@{->>}[d] & \Weq(\U) \ar@{->>}[d]^{(s,t)} \\
X \times X \ar[r]_{a \times a} & \U \times \U \, ,}
\]
where $a \co X \to \U$ is a classifying map for the small fibration $p \co A \to X$, which exists by our assumption
that $A$ and $X$ are cofibrant and part~(iii) of~\cref{thm:universe-u}. The verification that $\Weq(p)$
has the required universal property is an easy calculation, which we leave to the readers. There is an
evident map $i \co X \to \Weq(p)$ corresponding via the universal property of $\Weq(p)$ to the tuple of identity maps $(1_X, 1_X, 1_A,e)$ where $e$ is just the ``trivial'' weak equivalence structure on the identity, using the identity as both a left and a right inverse.

\begin{definition}  \label{equ:characterisations-of-univalence} Let $p \co A \to X$ be a small fibration with cofibrant fibers. We say that $p$ is \myemph{univalent} if the map $i \co X \to \Weq(p)$ is a weak equivalence. 
\end{definition}

\smallskip

For a small fibration with cofibrant fibers $p \co A \to X$, being univalent is equivalent to
either $s \co \Weq(p) \to X$ or $t \co \Weq(p) \to X$ being a trivial fibration.  Also note that, when $X$
is cofibrant, we have a map $j \co \Path(X) \to \Weq(p)$ fitting in the diagram
\[
\xymatrix@C=1.5cm{
X \ar[r]^i \ar@{>->}[d]_r & \Weq(p) \ar@{->>}[d]^{(s,t)} \\
\Path(X) \ar[r]_{\partial_X}  \ar@{.>}[ur]_{j} &  X \times X }
\]
In this case, $p$ is univalent if and only if $j$ is a weak equivalence, mirroring the type-theoretic
definition of univalence. This will be the case for $\pi_c \co \UU_c \to \U_c$, for example.

\begin{theorem}  \label{thm:univalence-of-u-and-uc} \hfill 
\begin{enumerate}[(i)]
\item The fibration $\pi \co \UU \to \U$ is univalent.
\item The fibration $\pi_c \co \UU_c \to \U_c$ is univalent.
\end{enumerate}
\end{theorem}

\begin{proof} For part (i), we prove that $t \co  \Weq(\U) \to \U$ has the right lifting property with respect
to all cofibrations. So let $f \co Y \rightarrow X$ be a cofibration and consider the diagonal
filling problem
\[
\xymatrix{Y \ar@{>->}[d]_f \ar[r] & \Weq(\U) \ar@{->>}[d] \\
X \ar[r] \ar@{.>}[ur]  & \U \, .
}
\]
By construction of $\Weq(\U)$, and using \cref{prop:Weq_classify_Weq}, this corresponds exactly to a diagram as in the equivalence extension property as in \cref{Prop:Homotopy_ext_prop}. Indeed, the map $X \rightarrow \U$ gives us
$p \co A \to X$, the composite of~$Y \rightarrow  \mathsf{Weq}(\U)$ with the first projection
gives us $q \co B \to Y$, while the rest of the data and the commutativity of the square 
gives us a weak equivalence $u$ between $B$ and $A[f]$ over~$X$. The completion of this diagram given by \cref{Prop:Homotopy_ext_prop} is exactly what one needs to produce the required diagonal filler.

For part~(ii), we prove that $t \co \Weq(\pi_c) \to \U_c$ is a trivial fibration. For this, let us
observe that we have a diagram 
\[
\xymatrix@C=1.5cm{
\Weq(\U_c) \ar@{->>}[r]^{\sigma}  \ar@{->>}[d]_{(s,t)} & \Weq(\U) \ar@{->>}[d]^{(s,t)} \\
\U_c \times \U_c \ar@{->>}[r]^{\tau \times \tau} \ar@{->>}[d]_{\pi_2} & \U \times \U \ar@{->>}[d]^{\pi_2} \\
\U_c \ar@{->>}[r]_\tau & \U \, . }
\]
The composite  on the left-hand side is the map $t$ that we wish to show to be a trivial fibration.
First, using part (i), observe that it is a fibration since it is the composite of two fibrations. Secondly,
observe that the top square is a pullback and so $\sigma$ is a trivial fibration since $\tau \times \tau$ 
is so. Thus, applying 3-for-2 to the outer square, we obtain that $t$ is a weak equivalence and hence
a trivial fibration, as required.
\end{proof}

\section{Conclusions and future work}
\label{sec:conclusion}

We can now sumarize our results using the comprehension category of cofibrant simplicial sets introduced in \cref{thm:compcat}, 
\[ 
\xymatrix{
\mathbf{Fib}_{ \mathsf{cof}} \ar[dr]_{p} \ar[rr]^{\chi} & & \SSet^{\to}_{\mathsf{cof}} \ar[dl]^{\mathrm{cod}} \\ 
 & \SSet_{\mathsf{cof}} \, . &  }
 \]
For this, we use the terminology of \cite{LumsdaineP:locuoc} regarding the stability conditions and
that of  \cite[Appendix~A]{ShulmanM:allths}  regarding the type-theoretic universe. When we say
that the universe is closed under some constructions, we simply mean that the fibrations classified by the universe are closed under these constructions.

\begin{theorem} \label{th:main_ContextualCat} The comprehension category of cofibrant simplicial sets has:
\begin{itemize}
\item a pseudo-stable empty type $0$, unit type $1$ and natural numbers type $\mathbb{N}$,
\item pseudo-stable $+$-types,
\item partially pseudo-stable $\Id$-types,
\item pseudo-stable $\Sigma$-types,
\item weakly stable $\Pi$-types,
\item a pseudo Tarski universe $\pi_c: \UU_c \rightarrow \U_c$, containing $0$, $1$, $\mathbb{N}$ and closed under the $\Sigma$-types, $\Id$-types, $\Pi$-types and $+$-types constructions.
\end{itemize}
Furthermore, the fibration $\pi_c: \UU_c \rightarrow \U_c $ is univalent.
\end{theorem}

\begin{proof}
We have checked that this is indeed a comprehension category in \cref{thm:compcat}. 
For $0$, $1$ and $\mathbb{N}$, these are simply given by the discrete simplicial sets, which are cofibrant (as every cell of dimension $>0$ is degenerated) and fibrant. For $+$ types, we use coproducts observing that
 $A + B$ is fibrant and cofibrant if $A$ and $B$ are so. 
 
Identity types are constructed in~\cref{sec:pats}. In particular, see \cref{thm:MainPathObject}. The construction of $\Id$-types themselves involves only categorical dependents products and pullbacks, which are all pseudo-stable, but the construction of the $\mathsf{J}$-eliminator involves a lifting properties which is not pseudo-stable 
in general. Thus, we have partially pseudo-stable identity types, as in~\cite[Definition 2.3.4]{LumsdaineP:locuoc}. The $\Sigma$-types have been constructed in \cref{thm:sigma-types} and are clearly stable under pullback up to isomorphism.  
The $\Pi$-types have been constructed in \cref{sec:Pi-types} as cofibrant replacement $\mathbb{L} \Pi_A B$ of the categorical dependent products $\Pi_A B$. The categorical dependant products are pseudo-stable, so given such a $\Pi$-type $\mathbb{L} \Pi_A B$, in context $\Gamma$ its pullback $f^*( \mathbb{L} \Pi_A B)$ along a context morphism $f \co \Gamma' \rightarrow \Gamma$ might be different from $\mathbb{L} \Pi_{f^* A} f^* B$ but is also a cofibrant replacement of $\Pi_{f^* A} f^* B$ hence also has the property of being a $\Pi$-type.

Finally, in \cref{sec:unifbb} we constructed the universe $\UU_c \rightarrow \U_c$ as a classifier for all small fibrations between cofibrant objects. As small fibrations are stable under all the constructor above, this gives the stability of the universe. It has been shown in \cref{thm:univalence-of-u-and-uc} in that this fibration is indeed univalent. Regarding the closure of the universe, the only delicate point concerns $\Pi$-types because of the
use of the cofibrant replacement.  The explicit definition of the cofibrant replacement functor 
in  \cref{sec:appendix} shows that  the trivial fibration $\varepsilon_X \co \mathbb{L} X \rightarrow X$ is a small fibration for every $X$, independently of whether $X$ is small or not (see part~(i) of \cref{cor:cofibrant_smallness} for details). Hence the construction of $\Pi$-types  preserves small fibrations, even between non-small objects.
\end{proof}

\begin{remark}
The $\Pi$-types constructed in the present paper have stronger properties than being weakly stable $\Pi$-types as stated above. First, as explained in \cref{rem:pi-types}, they satisfy a propositional $\eta$-rule, while \cite{LumsdaineP:locuoc} require no $\eta$-rule. Secondly, given our $\Pi$-type $\mathbb{L}(\Pi_A B)$ over $X$ in the sense of \cref{th:main_ContextualCat}, and $f \co X'  \rightarrow X$, then  both $f^*( \mathbb{L} ( \Pi_A B ))$ and $\Pi$-type $\mathbb{L} \big( \Pi_{f^* A} f^* B \big)$ are both cofibrant replacements of $f^* (\Pi_A B) \cong \Pi_{f^*A} f^* B $. Hence, they are homotopy equivalent. 
\end{remark}

\begin{remark} \label{rem:strength}
Our proof that the universe is closed under $\Pi$-types relies on the assumption that the inacessible set $\mathsf{u}$ satisfies a form of propositional resizing, as observed in~\cref{cor:cofibrant_smallness}. If one wishes to avoid this additional assumption, there are at least two alternatives.
The first alternative arises if one does not make the assumption that subsets of small sets are small and keeps all the definitions the same. Then, we still have a comprehension category with all the structure mentioned in~\cref{th:main_ContextualCat}, but do not have anymore that the universe is closed under the construction of arbitrary $\Pi$-types. Instead, one has (in type-theoretic notation) that if  $\Gamma, a \co A \vdash B(a) \co \U_c$ then $\Gamma \vdash \Pi_A B \co \U_c$ if $\Gamma$ is a context  involving only small types. The second alternative arises if one works instead in CZF without assuming any inacessible sets, and changes to an interpretation where contexts are ``simplicial classes'' (where classes are formulas), general types are given by Kan fibrations between simplicial classes and a small fibrations is one whose fibers are sets. In this case, part~(ii) of \cref{cor:cofibrant_smallness}  provides the smallness property that ensures that the cofibrant replacement preserves small type even in non-small context.  We then   have that the universe is closed under $\Pi$-type, but we cannot form general  $\Pi$-types, but only those of the form $\Pi_A B$ where $A$ is a small type.
\end{remark}

\begin{remark} A lack of strict stability for $\Pi$-types is also present in the semisimplicial
model defined in~\cite{CoquandT:gentgi}. This is of interest since cofibrant simplicial
sets are closely connected to semisimplicial sets~\cite{henry2019qms}. 
\end{remark} 

\begin{remark}
The main result of \cite{LumsdaineP:locuoc} assert that if a contextual category $\mathcal{C}$, satisfying 
the additional condition  (LF) of \cite[Definition 3.1.3]{LumsdaineP:locuoc}) has weakly stable type constructors, then left adjoint splitting $\mathcal{C}_!$ has the same constructors, but now strictly stable,
thus solving the coherence problem. Unfortunately, condition (LF) is not satisfied in our constructive
setting, essentially because cofibrant objects are generally not closed under dependent products.  
\end{remark}

We leave as an open problem the question of whether a comprehension category as in~\cref{th:main_ContextualCat} can be split so as obtain a model of ML(UA) and so, in particular,
obtain a constructive version of the simplicial model of ML(UA). It should be noted that this question
is now completely independent from simplicial homotopy theory, as our results in this paper have 
obtained all the necessary results to produce the structure in~\cref{th:main_ContextualCat}. 
We also leave as open problem the question of whether our comprehension category supports
other inductive types~\cite{vandenberg_moerdijk_2015} and higher-inductive types in the sense of~\cite{hottbook}. 
We expect these to work without the need to perform cofibrant replacements.

Finally, we mention as a potential area of further research the idea of designing a dependent type
theory with explicit substitutions~\cite{AbadiM:exps} with rules matching the structure of the comprehension category of~\cref{th:main_ContextualCat}, extending the idea of 
substitution up to isomorphism of~\cite{CurienP:subi} to that of substitution up to weak equivalence. 
A conservativity result of such a system over ML(UA)
would then be essentially equivalent to a solution to the coherence problem, {cf.}~\cite{CurienPL:revcit}.

\appendix

\section{The cofibrant replacement functor}
\label{sec:appendix}

The goal of this appendix is to give an explicit description of the cofibrant replacement functor on simplicial sets (\cref{prop:Cof_replacement}) and to use it to discuss the size of this cofibrant replacement.

\begin{definition}\hfill
\begin{enumerate}[(i)]  
\item If $x \in X_n$ is an $n$-cell in a simplicial set $X$, the \emph{degeneracy type} of $x$ is the set of face maps $\delta:[i] \rightarrow [n]$ such that $\delta^* x$ is a degenerate cell.
\item By a \emph{degeneracy type} or \emph{degeneracy $n$-type} we mean a subset of faces of $\Delta[n]$ that can appear as the degeneracy type of an $n$-cell of a simplicial sets.
\end{enumerate}
\end{definition}

\medskip

Note that given any $n$-cell $x \in X_n$, the factorization of $x \co \Delta[n] \rightarrow X$ as a degeneracy quotient followed by a degeneracy detecting morphism constructs a degeneracy quotient $y \co \Delta[n] \rightarrow K$ such that $y$ has the same degeneracy type as $x$. This induces a correspondence between degeneracy $n$-types and (isomorphism classes of) degeneracy quotients $\Delta[n] \rightarrow K$. 

\begin{lemma} A degeneracy $n$-type is the degeneracy type of a cell in a cofibrant object  if and only if it is decidable as a subset of faces of $\Delta[n]$.
\end{lemma}

\begin{proof}
Given an $n$-cell $x$ in a cofibrant simplicial set $X$, given any face $\delta:[m] \rightarrow [n]$ it is decidable if $\delta^* x$ is degenerated or not, hence the degeneracy type of $x$ is decidable. Conversely, given any degeneracy type $p$, it is the degeneracy type of a degeneracy quotient $\Delta[n] \rightarrow K$. Any cell of $K$ is the image of a cell of $\Delta[n]$, and it is degenerated in $K$ if and only if it is the image of a degenerated cell, or the image of a cell in $p$. Hence if $p$ is decidable $K$ is cofibrant.
\end{proof}

We will now construct a ``simplicial set of decidable degeneracy type $D$''. Note that the construction would work exactly the same without the assumption ``decidable'', with the exception that as CZF does not has power set, this might not be a simplicial set, but rather a simplicial class.

 We start with a simplicial set $D_0$ whose $n$-cells are all the decidable subset of the faces of $\Delta[n]$, i.e. decidable subsets of the set of finite subsets of $[n]$. As decidable subsets are justs functions to $\{0,1\}$ this is indeed a set.  If $f : [m] \rightarrow [n]$ and $P \in D_0([n])$ one defines $f^* P$ as the set of faces  $[i] \rightarrow [m]$ such that the composite 
 \[
 [i] \rightarrow [m] \overset{f}{\rightarrow} [n]
 \] is either in $P$ or non-injective, which is indeed a decidable set of faces.

\begin{lemma} \label{lem:prev}
$D_0$ is a cofibrant simplicial set, and for any cofibrant simplicial set $X$, the map $X \rightarrow D_0$ sending each $n$ cell to its degeneracy type is a simplicial and degeneracy detecting map.
\end{lemma}

\begin{proof}
It is immediate to check that $D_0$ is a simplicial set. Moreover as if has decidable equality one can decide whether a cell $x$ is degenerated by testing if it is equal to $s^* d^* x$ for some non trivial face $d$ and degeneracy $s$, hence $D_0$ is cofibrant. Let $x \in X_n$ and $f:[m] \rightarrow [n]$, in order to check that the maps $X \rightarrow D_0$ sending each cell to its degeneracy type is simplicial, one need to check that the degeneracy type of $f^* x$ is indeed described from the degeneracy type of $x$ using the formula for the functoriality of $D_0$. For face maps $[i] \rightarrow [m]$ the cell $i^* f^* x$ is degenerated as soon as $f \circ i$ is non injective, and if $f\circ i$ is injective, then it depends on whether $f \circ i$ is in the degeneracy type of $x$ or not.

If $f:[n] \rightarrow [m]$ is non-injective then for any $s \in (D_0)_m$ the identity map $[n] \rightarrow [n]$ is in $f^* s$. So a degenerate cell of $D_0$ always contains the maximal face. In particular the map $X \rightarrow D_0$ constructed above send non-degenerate cell of $X$ to non-degenerate cells of $D_0$.
\end{proof}

\cref{lem:prev} constructs a map $P \co D_0 \rightarrow D_0$ sending any cell to its degeneracy type. As $P$ preserves and detects degeneracy, it preserves the degeneracy type and hence $P \circ P =P$.

\begin{definition}
The simplicial set $D$ is the set of fix point of the idempotent $P$ acting on $D_0$.
\end{definition}

Note that for any simplicial set $X$, the map $f:X \rightarrow D_0$ sending each cell to its degeneracy type preserves the degeneracy type, hence $P f = f$. In particular a cell of $D_0$ is in $D$ if and only if it is a degeneracy type. This leads to the following fact. 

\begin{lemma}
$n$-cells of $D$ are in bijection with decidable degeneracy $n$-type. Each cofibrant simplicial set $X$ has a unique map detecting degeneracy to $D$, which is the map sending a cell to its degeneracy type. \qed
\end{lemma}

\begin{lemma} \label{lem:D_contractible}
Given a map $\partial \Delta[n] \rightarrow D$ there is a unique way to extend it into map $\Delta[n] \rightarrow D$ such that the non-degenerate $n$-cell of $\Delta[n]$ is sent to a non-degenerate cell.
\end{lemma}

\begin{proof}
If such an extension exists the $n$-cell $x$ corresponding to $\Delta[n] \rightarrow D$ has to be as follows: $x$ does not contains the face $[n] \rightarrow [n]$, and for all other face $\delta:[i] \rightarrow [n]$ it is in $x$ is and only the composite $[i] \rightarrow \partial \Delta[n] \rightarrow D$ is a degenerate cell. So the uniqueness is clear. We only need to show that this set is indeed a degeneracy type. But if one form $D \coprod_{\partial \Delta[n]} \Delta[n]$ then the new added $n$-cell has exactly this degeneracy type so this conclude the proof.
\end{proof}

Let $\varepsilon_X \co \mathbb{L} X \to X$ be the cofibrant replacement of a simplicial set $X$ constructed using  Garner's version of the small object argument~\cite{garner:small-object-argument}.

\begin{proposition} \label{prop:Cof_replacement}
The map $\mathbb{L} X \rightarrow D \times X$ sending an $n$-cell to its degeneracy type $s \in D$ and its image in $X$ induces a bijection between $\mathbb{L}X([n])$ and the set of pairs of a decidable degeneracy type $s$ and a cell of $x$ of degeneracy type larger than $s$.
\end{proposition}

\begin{proof} Because of the stratified nature of the simplicial generating cofibration, $\mathbb{L} X$ can be written as the colimit of a sequence:
 \[ 
 \mathbb{L}^{(0)} X \hookrightarrow \mathbb{L}^{(1)} X \hookrightarrow \dots \hookrightarrow \mathbb{L}^{(n)} X \hookrightarrow \dots 
 \]
where $\mathbb{L}^{(0)} X$ is just the set of $0$-cell of $X$ and the map $\mathbb{L}^{(n-1)} X \rightarrow \mathbb{L}^{(n)} X$ is constructed as the pushout of the coproduct of one copy of $\partial \Delta[n] \rightarrow \Delta[n]$ for each square of the form:
\[ 
\xymatrix{\partial \Delta[n]  \ar[r] \ar@{>->}[d] & \mathbb{L}^{(n-1)} X \ar[d] \\ \Delta[n]  \ar[r] & X} 
\]
Indeed, the pushout constructing $\mathbb{L}^{(n)} X$ from $\mathbb{L}^{(n-1)} X$ do not change the $k$-skeleton for $k <n$ and the set of maps $\partial \Delta[k] \rightarrow \mathbb{L}^{(n-1)} X$ only depends on the $k$-skeleton of $X$, so one can always do all the necessary pushout by $\partial \Delta[k] \hookrightarrow \Delta[k]$ by for $k<n$ before those by $\partial \Delta[n] \hookrightarrow \Delta[n]$.

We will prove by induction on $n$ that $\mathbb{L}^{(n)} X$ identifies with $n$-skeleton of the simplicial set mentioned in the proposition, i.e. the sub-simplicial set $Y^n$ of $D \times X$ of pairs $(s,x)$ such $s$ is decidable,  the degeneracy type of $x$ is at least $s$ and $(s,x)$ is of dimension at most $n$, or a degeneracy of a cell of dimension at most $n$.

Note that as the degeneracy type of $x$ is larger than $s$, $s$ is the degeneracy type of the pair $(s,x)$, so the condition on $(s,x)$ being a degeneracy is equivalent to the same condition on $s$.

In the case $n=0$, there is only one degeneracy type in dimension $0$, so both $\mathbb{L}^{(0)} X $ and $Y^0$ are the simplicial set of cells of $X$ that are degeneracy of $0$-cells.

Assuming $Y^{n-1}$ and $\mathbb{L}^{(n-1)} X$ are isomorphic as claimed. One only needs to check that the new non-degenerate $n$-cell that appears in $Y^{n}$ and $\mathbb{L}^{(n)} X $ are in bijection compatible to their boundary and their image in $X$.
In $\mathbb{L}^{(n)} X $ there is exactly one such cell for each square as above. In $Y^{n}$, any non-degenerate $n$-cells does produce such a square, and conversely given a square:

\[ \xymatrix{\partial \Delta[n]  \ar[r] \ar@{^{(}->}[d] & Y^{n-1} \ar[d] \\ \Delta[n]  \ar[r] & X} \]

\cref{lem:D_contractible} above gives a unique map to extend $\partial \Delta[n] \rightarrow D$ to a non-degenerate $n$-cell of $D$, and the image of $\Delta[n]$ in $X$ automatically has a larger degeneracy type that this extension so this gives a non-degenerate cell of $Y^{n}$ generating this square. This is the unique way to get such a cell in $Y^n$ to be non-degenerate: indeed a cell in $Y^n$ is non-degenerate if and only if its image in $D$ is non-degenerate.
\end{proof}

\begin{corollary}
The simplicial set $D$ of decidable degeneracy type is the cofibrant replacement $\mathbb{L} 1$ of $1$.
\end{corollary}

\begin{corollary} \label{cor:cofibrant_smallness}
\begin{enumerate}[$(i)$]

\item[]
 
\item Under the assumption that ``subsets of small sets are small sets'', the map $\mathbb{L} X \rightarrow X$ is a small fibration for any simplicial set $X$.

\item In $\mathrm{CZF}$, given a simplicial class $X$, the cofibrant replacement $\mathbb{L} X \rightarrow X$ exists as a simplicial class, is cofibrant, and the class map $\mathbb{L} X \rightarrow X$ is a trivial fibrations whose fibers are sets.

\end{enumerate}
\end{corollary}

\begin{proof}

For part~$(i)$, the cells of $\mathbb{L} X$ over a given cell $x \in X([n])$ are the decidable degeneracy $n$-types $s$ smaller than the degenracy type of $x$. In particular this identifies with a subset of the set of all degeneracy $n$-types and hence form a small set by assumption.

For part~$(ii)$, the explicit description of $\mathbb{L} X$ given by \cref{prop:Cof_replacement} immediatly allows to make sense of it as a simplicial class. Indeed the class of $(s,x)$ where $x\in X_n$, $s$ is a decidable degeneracy type smaller than the degeneracy type of $x$ is clearly a type. Such a pair $(s,x)$ is degenerated if and only if $s$ is, hence this is decidable, and using \cref{lem:D_contractible} one can immediately see that the projection $\varepsilon_X \co \mathbb{L} X \rightarrow X$ is a trivial fibrations. Finally, the fiber over an $n$-cell $x \in X$, is isomorphic to a subsets of the set of decidable degeneracy $n$-types defined by Restricted Separation, hence it is a set.
\end{proof}

\newcommand{\noopsort}[1]{}

\end{document}